\documentclass[11pt]{amsart}

\usepackage[dvipsnames]{xcolor}
\usepackage[english]{babel}
\usepackage[T1]{fontenc}
\usepackage[utf8]{inputenc}
\usepackage{amsmath}
\usepackage{amssymb}
\usepackage{amsthm}
\usepackage{color}
\usepackage[a4paper]{geometry}
\usepackage{tikz}
\usepackage[all]{xy}
\usepackage{hyperref}
\usepackage{lmodern}
\usepackage{mathrsfs}

\theoremstyle{definition}
\newtheorem{definition}{Definition}[section]

\newtheorem{rmk}[definition]{Remark}

\newtheorem{quest}[definition]{Question}

\theoremstyle{plain}
\newtheorem{theorem}[definition]{Theorem}
\newtheorem{prop}[definition]{Proposition}
\newtheorem{cor}[definition]{Corollary}
\newtheorem{lem}[definition]{Lemma}

\newcommand{\UA}{\operatorname{D}}

\newcommand{\mb}{\mathbb}
\newcommand{\mc}{\mathcal}
\newcommand{\mf}{\mathfrak}

\newcommand{\bs}{\boldsymbol}

\newcommand{\Lie}{\operatorname{Lie}}

\newcommand{\vol}{\operatorname{vol}}
\newcommand{\dist}{\operatorname{dist}}
\newcommand{\SL}{\operatorname{SL}}
\newcommand{\R}{\mathbb{R}}
\newcommand{\s}{\mathbb{S}}
\newcommand{\Z}{\mathbb{Z}}
\newcommand{\N}{\mathbb{N}}
\newcommand{\dd}{\textup{d}}
\newcommand{\id}{\operatorname{id}}

\newcommand{\Lone}{\operatorname{L}^{1}}
\newcommand{\da}{Diophantine approximation}
\newcommand{\ul}{\underline}

\makeatletter
\newcommand{\vast}{\bBigg@{3}}
\newcommand{\Vast}{\bBigg@{4}}
\makeatother

\newcommand{\ignore}[1]{}

\newif\ifdraft\drafttrue

\draftfalse

\newcommand\eq[2]{{\ifdraft{\ \tt [#1]}\else\ignorespaces\fi}\begin{equation}\label{eq:#1}{#2}\end{equation}}
\newcommand {\equ}[1]     {\eqref{eq:#1}}

\numberwithin{equation}{section}

\usepackage{scalerel,stackengine}
\stackMath
\newcommand\reallywidehat[1]{%
\savestack{\tmpbox}{\stretchto{%
  \scaleto{%
    \scalerel*[\widthof{\ensuremath{#1}}]{\kern-.6pt\bigwedge\kern-.6pt}%
    {\rule[-\textheight/2]{1ex}{\textheight}}
  }{\textheight}%
}{0.5ex}}%
\stackon[1pt]{#1}{\tmpbox}%
}
\begin{document}

\title[Submanifold-genericity and uniform multiplicative approximation]{Submanifold-genericity of $\R^d$-actions and\\  uniform multiplicative Diophantine approximation}

\author{Prasuna Bandi}
\address{Department of Mathematics, University of Michigan, Ann Arbor MI}
\email{prasuna@umich.edu}

\author{Reynold Fregoli}
\address{Department of Mathematics, University of Michigan, Ann Arbor MI}
\email{reynoldfregoli@gmail.com}

\author{Dmitry Kleinbock}
\address{Department of Mathematics, Brandeis University, Waltham MA}
\email{kleinboc@brandeis.edu}

\thanks{The third named author was supported by NSF grant DMS-2155111.}
\date{March 2025}

\subjclass{37A25, 37A44, 11J13, 11K60}

\begin{abstract}
In this paper, we prove a new ergodic theorem for $\mb R^d$-actions involving averages over dilated submanifolds, thereby generalizing the theory of spherical averages. Our main result is a quantitative estimate for the error term of such averages valid for smooth functions under some effective mixing assumptions on the action. {With the aid of this theorem, we investigate  multiplicative-type Dirichlet-improvability for $(m\times n)$-matrices with real coefficients.} In particular, we
establish that almost all matrices are uniformly approximable by the function $x\mapsto x^{-1}(\log x)^{-1+\varepsilon}$ for any $\varepsilon>0$. Results of this type motivate a {question} which can be thought as a strengthening of Littlewood's conjecture in multiplicative \da.
\end{abstract}

\maketitle

\section{Introduction}

\subsection{Averages along dilates of submanifolds}

Let  $(X,\mu)$ be a probability space. Fix $d\in\N$ and  consider a $\mu$-preserving action $(a,x)\mapsto a.x$ of $\mb R^d$ on $X$, where $a\in \mb R^d$ and $x\in X$. For a subset $B$ of $\R^d$ let $tB$  denote the dilate of $B$ by a real parameter $t>0$ and let $\vol_d$ stand for the Lebesgue measure on $\R^d$. A simple consequence of Birkhoff's Ergodic Theorem is that the action $(a,x)\mapsto a.x$ is ergodic if and only if for any function $\varphi\in \Lone(X,\mu)$ and $\mu$-a.e.\ $x\in X$ it holds that
\begin{equation}
\label{eq:Birk-B}
\frac{1}{t^d\vol_d(B)}\int_{tB}\varphi(a.x)\,\dd\vol_d(a)\to \int_{X}\varphi\,\dd\mu\quad  \text{as }t\to \infty,   
\end{equation}
where $B$ is the Euclidean unit ball in $\R^d$. In fact, more generally, $B$ in \eqref{eq:Birk-B} can be any bounded positive-measure subset of $\R^d$ (see e.g. \cite[Theorem 1.2]{Lindenstrauss}), while the test functions $\varphi$ may be chosen from a dense subspace $\mc F$ of $\Lone(X,\mu)$. For example, if $X$  is a smooth manifold and $\mu$ is a volume measure, it is enough to assume that $\varphi$ is smooth and compactly-supported.

The convergence in \eqref{eq:Birk-B} becomes substantially more delicate if
the measure $\vol_d$ in the left-hand side is replaced by a singular measure. Such averages have been studied at length, especially in the case when $B=\mb S^{d-1}$ (spherical averages); see \cite{Stein, Bourgain, Jones, Lacey, Jaming}. Notably, in this case, the almost sure convergence in \eqref{eq:Birk-B} depends on the regularity of the test functions to which the averaging is applied. 

This suggests the following. 

\begin{definition}\label{def:genericity} Take $k\le d$, let $\mathcal F$ be a subspace of $\Lone(X,\mu)$, and let $M$ be a 
{compact} $k$-dimensional $\mathscr{C}^1$  submanifold {(possibly with boundary, which in what follows will be a standing assumption on compact manifolds $M$)} of $ \R^d$. Denote by $\vol_k$ the $k$-dimensional volume measure on $M$ induced by the Euclidean metric on $\R^d$, and let $\nu$ be a probability measure on the the space $X$ (possibly equal to $\mu$). Let us say that the measure $\nu$ is {\sl $(M,\mathcal F)$-generic} if for any function $\varphi \in \mathcal F$ and for $\nu$-a.e.\ point $x\in X$ one has that
\begin{equation}\label{eq:Birk-S}
\frac1{t^{k}\vol_k(M)}\int_{tM}\varphi(a.x)\,\dd\vol_k(a)\to\int_X\varphi\,\dd\mu\quad \text{ as }t\to\infty.
\end{equation}
Let us also say that, for $E\subset \R^d$, the measure $\nu$ is \textsl{$(E,k,\mathcal F)$-generic} if it is $(M,\mathcal F)$-generic for any $k$-dimensional {compact} $\mathscr{C}^1$ submanifold
$M$
contained in 
$E$. If $E=\mb R^d$, we will simply say that $\nu$ is \textsl{$(k,\mathcal F)$-generic}.
\end{definition}

As a motivation, recall that when $X$ is compact, a point $x\in X$ is said to be generic for the action of a group $G$ if \eqref{eq:Birk-S}, with $k =d$ and $M$  a unit ball in $\R^d$, holds for any continuous function $\varphi$ on $X$. Hence it seems natural to extend this definition from points to measures, so that a point $x\in X$ is generic if and only if so is the Dirac measure $\delta_x$. This analogy, however, comes with a caveat: we do not assume the existence of a full measure set of points $x\in X$ (with respect to $\nu$) satisfying \eqref{eq:Birk-S} for all functions $\varphi\in\mc F$ simultaneously.

\begin{rmk}
Here are a few comments on the above definition:
\begin{itemize}
\item[(i)]
As was mentioned before, $(d,\mathcal F)$-genericity of the measure $\mu$ itself is equivalent to ergodicity of the action whenever $\mathcal F$ is dense in $\Lone(X,\mu)$. Furthermore, $(k,\mathcal F)$-genericity of $\mu$ for some $k\le d$ and some subspace $\mathcal{F}$ dense in $\Lone(X,\mu)$ implies ergodicity, simply because it implies ergodicity for the action of any $k$-dimensional subspace of $\R^d$. In view of this, the Haar measure on $\R^d/\Z^d$  is not $(k,\mathcal F)$-generic with respect to the standard (transitive $\Rightarrow$ ergodic) action of $\R^d$ on $\R^d/\Z^d$ for any $\mathcal F$ dense in $\Lone(\R^d/\Z^d)$ and any $k<d$.
Also, clearly, no measure is $(0,\mathcal F)$-generic if the action is non-trivial (and $\mathcal F$ is large enough).
\smallskip
\item[(ii)]  
Consider now the special case of $M=\s^{d-1}$, the unit sphere in $\R^d$. Then there do not exist any non-trivial $\R^d$-actions on $(X,\mu)$ such that $\mu$ is $\big(\s^{d-1},\Lone(X,\mu)\big)$-generic; that is, spherical integrals cannot converge for all functions in $\Lone$. However, convergence occurs if $\varphi\in \operatorname{L}^{p}$ for large enough $p$. For example, according to \cite[Theorem 1.2]{Jaming}, based on earlier work of Jones and Lacey \cite{Jones, Lacey}, for any ergodic $\R^d$-action on $(X,\mu)$, where $d\ge 2$,  the measure $\mu$ is $(\s^{d-1},\operatorname{L}^{p}(X,\mu)\big)$-generic provided $p >\frac d {d-1}$. 
In fact, this range of $p$ is best possible. For example, given any ergodic, measure-preserving, and aperiodic $\R^d$-action on a probability space $(X,\mu)$, for every $p\leq \frac{d}{d-1}$ there exists a function $f\in L^{p}(X,\mu)$ such that \eqref{eq:Birk-B}, with $B$ replaced by $ \s^{d-1}$, fails for $\mu$-a.e.\ $x\in X$ \cite[Theorem 2.3]{Jones}.
\end{itemize}
\end{rmk}

The preceding remark shows that $\R^d$-actions such that the measure $\mu$ itself is $(k,\mathcal F)$-generic for some choice of $k<d$ and dense $\mc F\subset \textup{L}^1(X,\mu)$ are rather hard to find. Our first result provides a fairly large class of such actions.

\begin{theorem}\label{thm:gmodgamma}
Let $G$ be a semi-simple Lie group of real rank $d$ whose simple factors have rank at least $2$. Let $\Gamma$ be a lattice in $G$ and let $\mu$ denote the unique $G$-invariant Radon probability measure on the space $X:=G/\Gamma$. Fix a Cartan subgroup $A$ of $G$, put $\mathfrak a := \Lie(A)$, and consider the action of $\mathfrak a\cong \mb R^d$ on $X=G/\Gamma$ given by \begin{equation}\label{eq:action}a.x := \exp(a)x\;\text{ for } a\in\mathfrak a.\end{equation}
Finally, let $\mathcal{F}:=\mathscr{C}^{\infty}_c(X)$ denote the space of compactly supported smooth functions on $X$. Then the measure
$\mu$ is $(k,\mathcal{F})$-generic for any $1\le k \le d-1$. Furthermore,
for any  {compact} $k$-dimensional $\mathscr{C}^1$ submanifold 
$M\subset \mf a$, any $\delta > 0$, any $\varphi \in \mathcal{F}$, and $\mu$-almost every $x\in X$ there exists $T_0  = T_0(\delta,\varphi,M,x) > 0$ such that
\begin{equation}\label{eq:Birk-S-eff}
\left|\frac1{t^{k}\vol_k(M)}\int_{tM}\varphi(a.x)\,\dd\vol_k(a)-\int_X\varphi\,\dd\mu\right| \leq t^{-k/2+\delta}  
\end{equation}
for all   $t\geq T_0$.
\end{theorem}

Our next result establishes the genericity of a special class of measures supported on closed unipotent orbits of $G/\Gamma$ with $G=\SL_{d+1}(\R)$, $d \ge 2$, and $\Gamma$ --- a lattice in $G$. The case $\Gamma =\SL_{d+1}(\Z)$ is important for our main Diophantine application, since then the homogeneous space $X=G/\Gamma$ can be identified with the space of unimodular lattices in $\R^{d+1}$.

Fix $m,n\in\N$ such that $m+n = d+1$ and put
\begin{equation}\label{eq:umn}
    U_{m,n}:=\left\{\begin{pmatrix}
        I_{m}& Y\\
        0& I_{n}
    \end{pmatrix}: Y\in \R^{m\times n} \right\}\subset G.
\end{equation}
Let $A\subset G$ be the group of diagonal matrices and consider the following subset of $\mathfrak a = \Lie (A)$:
\begin{equation}
\label{eq:E}
E:=\big\{ a = (a_1,\dots,a_{m+n}) \in {\mathfrak a} : a_1,\dots,a_m > 0,\  a_{m+1},\dots,a_{m+n} < 0\big\}.    
\end{equation}
Then for the action of $\mathfrak a$ on $X$ in \eqref{eq:action} the following holds.

\begin{theorem}
\label{cor:diophantine} Let $\nu$ be the unique $U_{m,n}$-invariant Radon probability measure supported on a given compact orbit of $U_{m,n}$ in $X$. Then $\nu$ is $\big(E,k,\mathscr{C}_{c}^{\infty}(X)\big)$-generic for any $1\leq k \leq d-1$. Furthermore, for any {compact} $k$-dimensional $\mathscr{C}^1$ submanifold
$M\subset E$, any $\delta>0$, and any $\varphi\in \mathscr{C}_{c}^{\infty}(X)$, \eqref{eq:Birk-S-eff} holds for $\nu$-a.e $x\in X$ and $t\geq T_0=T_0(\delta, \varphi, M, x)$.  
 \end{theorem} 

In the next subsection we present an application of Theorem \ref{cor:diophantine} to Diophantine approximation.
 
\subsection{Uniform approximation in the multiplicative set-up}
 For positive integers $m,n$,
 a function $\psi:{[1,\infty)\to (0,\infty)}$, and 
 $T> 1$ consider the following system of inequalities: \begin{equation}
\label{eq:systempsi}
\begin{cases}
\max_{i=1}^{m}|Y_{i}\bs{q}-p_{i}|^{m} < \psi(T) \\
\max_{j=1}^{n}|q_j|^{n} < T
\end{cases},
\end{equation}
where $Y_{i}$ for $i=1,\dotsc,m$ denote the rows of the matrix $Y\in\R^{m\times n}$. Let
\begin{equation}
\label{eq:adds}
{\mc{S}_{m,n}(
\psi,T):= \\
\left\{Y\in \mb{R}^{m\times n}:\exists\,\bs{p}\in\mb{Z}^{m},\, \bs{q}\in\mb{Z}^{n}\setminus\{\bs{0}\}
\text{ such that }\eqref{eq:systempsi}\text{ holds}
\right\}\nonumber}
\end{equation}
and let $\psi_1$ denote the function $x\mapsto 1/x$ for $x\in(1,\infty)$. 
Dirichlet's Theorem in Diophantine approximation states that
\begin{equation}
\label{eq:Dirt}
\mc{S}_{m,n}\left(c\psi_1,T\right)=\mb R^{m\times n}\quad\mbox{for all }c>1\mbox{ and }T>1;
\end{equation}
in other words, the system 
\begin{equation}
\label{eq:system}\tag{$c$-DI}
{\begin{cases}
\max_{i}|Y_{i}\bs{q}-p_{i}|^{m} < c/T\\
\max_{j}|q_j|^{n} < T
\end{cases}}
\end{equation}
has a non-trivial integer solution for any $Y\in \mb{R}^{m\times n}$ and any $T > 1$, as long as $c > 1$.
This result is known to be optimal in the following sense: if $c<1$,
the set of $c$-\textsl{Dirichlet-Improvable} matrices
\begin{equation*}
\label{eq:dimpmeas0}
\begin{aligned}
\liminf_{T\to\infty} \mc{S}_{m,n}\left(c\psi_1,T\right) &=      \bigcup_{{T_0 > 1}}\bigcap_{T\geq T_0}\mc{S}_{m,n}\left(c
\psi_1,T\right)    \\ 
& = \left\{Y\in \mb{R}^{m\times n} \left|\begin{aligned} \eqref{eq:system} \text{ has a non-trivial integer }\\\text{solution for all \textsl{large enough} }T\ \end{aligned}\right.\right\}
\end{aligned}
\end{equation*}
has Lebesgue measure zero \cite{DS2, KW08}. 
More generally, for a function $\psi:(1,\infty)\to[0,\infty)$ one may introduce the notion of a \textsl{$\psi$-Dirichlet-Improvable} or \textsl{$\psi$-Dirichlet} matrix $Y\in \operatorname{D}_{m,n}(\psi)$ (see \cite[\S 4]{KW18}), where
$$
\operatorname{D}_{m,n}
(\psi):=\liminf_{T\to\infty} \mc{S}_{m,n}\left(\psi,T\right)=      \bigcup_{{T_0 > 1}}\bigcap_{T\geq T_0}\mc{S}_{m,n}\left(\psi,T\right).$$
In \cite{KW18, KSY22}, {under some additional monotonicity assumptions on $\psi$,} it was shown that, depending on the convergence or the divergence of a certain sum defined in terms of the function $\psi$, the set of $\psi$-Dirichlet matrices has either full or zero Lebesgue measure in $\mb R^{m\times n}$ (although for $\max\{m,n\}>1$ the dichotomy is not complete). {Such statements are known in the literature as "Khintchine-type theorems"}. A related notion is that of singular {matrices}. $Y\in \mb R^{m\times n}$ is said to be \textsl{singular} (notation: $Y\in\operatorname{Sing} _{m,n}$) if it is {$c\psi_1$}-Dirichlet-Improvable for any $c > 0$. In other words, one has that
$$\operatorname{Sing} _{m,n} := \bigcap_{c > 0}\UA_{m,n}
(c\psi_1) =\bigcap_{c > 0} \left\{Y\in \mb{R}^{m\times n} \left|\begin{aligned} \eqref{eq:system} \text{ has a non-trivial integer }\\\text{solution for all {large enough} }T\ \end{aligned}\right.\right\}.$$
The set $\operatorname{Sing}_{m,n}$ has been the object of extensive work. {It coincides with $\mb Q$ when both $m$ and $n$ are equal  to $1$}, and is known to have Hausdorff dimension equal to $mn-\frac{mn}{m+n}$ {when $\max\{m,n\} > 1$} \cite{Che11,CC16,KKLM17,DFSU}.  

The notions of Dirichlet-improvability and singularity belong to the area of \textsl{uniform Diophantine approximation}. The word "uniform" here signifies that the sets under consideration are defined by "liminf conditions". The counterpart area of \textsl{asymptotic approximation} concerns instead "limsup sets", such as
\begin{equation}\label{eq:asymp}
\nonumber
\begin{aligned}
\operatorname{W}_{m,n}(\psi)&:=\limsup_{T\to\infty} \mc{S}_{m,n}\left(\psi,T\right) \\
&= \left\{Y\in \mb{R}^{m\times n}\ \left|\begin{aligned}
\ \eqref{eq:systempsi}\text{ has a nontrivial integer solution}\\ \text{for an \textsl{unbounded} set of }T>1\quad\end{aligned}\right.
\right\}.
\end{aligned}
\end{equation} 
{Note that, under the assumption that \eq{monotonicity}{\psi\text{ is continuous, non-increasing and satisfies }\lim_{T\to\infty}\psi(T) =0,} one can equivalently write 
\begin{equation*}
\operatorname{W}_{m,n}(\psi)=\left\{Y\in \mb{R}^{m\times n}\ \left|\begin{aligned}
\max_{i}|Y_{i}\bs{q}-p_{i}|^{m} &< \psi\left(
\max_{j}|q_j|^{n}\right)\\ \text{ for infinitely many }\bs{q}&\in\Z^n\text{ and some }\bs{p}\in\Z^m\end{aligned}\right.\right\}.
\end{equation*}}
These sets are part of the classical theory of Diophantine approximation, {with definitive results on Lebesgue measure established by Khintchine, Groshev, and Schmidt 
(the Khintchine--Groshev Theorem, see \cite{Gr38, schmidt})}.

While \eqref{eq:systempsi} and the related
sets $\UA_{m,n}(\psi)$ and $\operatorname{W}_{m,n}(\psi)$ are at present relatively well-understood, it appears substantially more difficult to 
study systems of inequalities defined instead in terms of products. {This gives rise to what is known in the literature as "multiplicative Diophantine approximation".}
Namely, for ${T> 1}$ and $\psi$ as above, one can define
\begin{equation}
{\mc{S}_{m,n}^{\times}(
\psi,T):= \\
\left\{Y\in \mb{R}^{m\times n}:\exists\,\bs{p}\in\mb{Z}^{m},\, \bs{q}\in\mb{Z}^{n}\setminus\{\bs{0}\}\mbox{ s.t.}\begin{cases}
\prod_{i=1}^{m}|Y_{i}\bs{q}-p_{i}|< \psi(T) \\
\Pi_{+}(\bs{q})< T
\end{cases}
\right\}\nonumber},
\end{equation}
where $\Pi_{+}(\bs x):=\prod_{i=1}^{n}\max\{1,|x_i|\}$ for any $\bs x\in\mb R^n$, and consider the set of  \textsl{multiplicatively $\psi$-approximable} matrices
\eq{psima}{\operatorname{W}^{\times}_{m,n}(\psi):=\limsup_{T\to\infty} \mc{S}^{\times}_{m,n}\left(\psi,T\right).}
{Observe that, under Assumption \equ{monotonicity}, these can be equivalently defined as}
\begin{equation*}
{\operatorname{W}^\times_{m,n}(\psi)=\left\{Y\in \mb{R}^{m\times n}\ \left|\begin{aligned}
\prod_{i=1}^{m}|Y_{i}\bs{q}-p_{i}| &< \psi\big(
\Pi_{+}(\bs{q})\big)\\ \text{ for infinitely many }\bs{q}&\in\Z^n\text{ and some }\bs{p}\in\Z^m\end{aligned}\right.\right\}.}
\end{equation*}
Clearly, the sets $\operatorname{W}_{m,n}(\psi)$ and $\operatorname{W}^{\times}_{m,n}(\psi)$ are the same when $m=n=1$; thus, in what follows, we shall assume that $\max\{m,n\} > 1$. 
This way, instead of \eqref{eq:system} one is led to consider the following system of inequalities:
\begin{equation}
\label{eq:systemmult}\tag{$c$-DI${}^\times$}
\begin{cases}
\prod_{i=1}^{m}|Y_{i}\bs{q}-p_{i}| < c/T,\\
\Pi_{+}(\bs{q})< T.
\end{cases}.
\end{equation}
Such a modification has its origin in a question posed by J.E.\ Littlewood around the third decade of the last century, on whether the set of \textsl{multiplicatively well-approximable} matrices
\begin{equation}
\nonumber
\label{eq:littlewood}
\begin{aligned}
\operatorname{WA}^{\times}_{m,n}&:=\bigcap_{c>0}\operatorname{W}_{m,n}^{\times}(c\psi_1)
\\    
& = \bigcap_{c>0}\left\{Y\in \mb{R}^{m\times n}\  \left|\begin{aligned} \eqref{eq:systemmult} \text{ has a non-trivial integer }\quad\\\text{ solution for an unbounded set of }{T>1} \end{aligned}\right.\right\}\\
&     = \left\{Y\in \mb{R}^{m\times n}:\inf_{\bs{q}\in\Z^n\setminus\{\bs{0}\}}\left( \prod_{i=1}^{m} \dist(Y_{i}\bs{q},\Z)\cdot \Pi_{+}(\bs{q}) \right)= 0
\right\}
\end{aligned}
\end{equation}
coincides with the whole space $\mb R^{m\times n}$. It should be remarked that the original conjecture refers 
to the case $m=2$, $n=1$. 

{Multiplicative Diophantine approximation}
has seen renewed interest in recent years (see, e.g., \cite{Bug09,BV11,Bad13,BB22}), due especially to the significant breakthrough {concerning the set of exceptions to Littlewood's Conjecture} achieved in \cite{EKL06}. One of the main goals of this paper is to 
develop a theory of uniform multiplicative approximation,
which appears not to have been treated in the literature. To this end, we let $\UA^{\times}_{m,n}(\psi)$ denote the set of  $\psi$-\textsl{multiplicatively-Dirichlet} matrices, defined as
\eq{psidir}{\UA^{\times}_{m,n}(\psi):=\liminf_{T\to\infty} \mc{S}^{\times}_{m,n}\left(\psi,T\right) = \bigcup_{{T_0> 1}}\bigcap_{T\geq T_0}\mc{S}^{\times}_{m,n}(\psi,T).}
Observe that, trivially,
\begin{equation*}\label{containment}\mc{S}_{m,n}(\psi,T)\subset \mc{S}_{m,n}^{\times}(\psi,T)\text{ for all }T
.\end{equation*}
{Hence, by \eqref{eq:Dirt},} we have that
\begin{equation}
\label{eq:1imsees}
c>1\;\Longrightarrow\;\mc S^{\times}_{m,n}(c\psi_1,T)=\mb R^{m\times n}\text{ for all }T>1\;\Longrightarrow\;\UA^{\times}_{m,n}(c\psi_1)=\mb R^{m\times n}.
\end{equation}
{In other words}, an analog of Dirichlet's Theorem holds in the multiplicative set-up. Based on this observation, it is natural to ask whether \eqref{eq:1imsees} is optimal, both in the set-theoretic and in the measure-theoretic sense. More precisely:  

\begin{quest}
\label{quest:1}
Does the set $\UA^{\times}_{m,n}(c\psi_1)$ have a non-empty complement for any $c<1$?
\end{quest}

\begin{quest}
\label{quest:2}
Does the set $\UA^{\times}_{m,n}(c\psi_1)$ have Lebesgue measure zero for any $c < 1$?
\end{quest}

In this paper we answer both questions in the negative. Our first result in this sense reads as follows.

\begin{prop}
\label{prop:dno}
For all integers $m,n\geq 1$, 
$$\UA^{\times}_{m,n}(c\psi_1)=\mb R^{m\times n} \;\text{ for }\;c>\frac{m! n!}{m^m n^n}.$$
\end{prop}

We are unable to further improve upon the constant $(m!n!)/(m^m n^n)$ for arbitrary matrices $Y$ {(note that the bound is optimal when $m=n=1$)}. However, we can show that, {unless $m=n=1$,}  \eqref{eq:1imsees} is very far from optimal in the measure-theoretic sense. Indeed, the following theorem follows as an application of Corollary \ref{cor:diophantine}.
\begin{theorem}\label{thm:special}
Let $\max\{m,n\} > 1$. Then for any $c>0$
$$\operatorname{Leb}\big(\UA^{\times}_{m,n}(c\psi_1)\big)=1.$$ 
\end{theorem}
We also establish a 
{Khintchine-type theorem}
for special functions $\psi_{1,\lambda}$ as follows. 
\begin{theorem}
\label{thm:main}
Let $\max\{m,n\} > 1$, and let $\psi_{1,\lambda}(x):=x^{-1}\cdot(\log x)^{-\lambda}$ for $\lambda>0$. Then
$$\operatorname{Leb}\left(\UA^{\times}_{m,n}(\psi_{1,\lambda})\cap[0,1)^{mn}\right)=\begin{cases}
0 & \mbox{for }\lambda>m+n-2, \\
1 & \mbox{for }\lambda<m+n-2.
\end{cases}$$
\end{theorem}
\pagebreak

\begin{rmk} A few comments are in order:
\begin{itemize}
\item[(i)]
    Gallagher \cite{Gallagher} (for $n=1$) and later Sprind\v{z}uk \cite[Ch.\ 1, Th.\ 13]{Spr79} (for $n>1$) established a complete analog of {the Khintchine--Groshev Theorem} 
    for the sets $\operatorname{W}^{\times}_{m,n}(\psi)$.
    In particular, for the function $\psi_{1,\lambda}$ their results imply that the critical exponent $\lambda$ where the Lebesgue measure changes from 1 to 0 is $m+n-1$. However, in the uniform framework, Theorem \ref{thm:main} shows that this critical exponent is $m+n-2$. \vspace{2mm} 
    \item[(ii)] Our argument yields a result that slightly improves upon Theorem \ref{thm:main} in the case where the measure is zero. In particular, it shows that the set $\UA ^{\times}_{m,n}(\psi)$ has zero measure whenever $\psi(T)=o\left(T^{-1}(\log T)^{-(m+n-2)}\right)$.
\end{itemize}
\end{rmk}

In view of Theorem \ref{thm:special},
the multiplicative analog of the set of singular matrices, i.e., the set
\begin{equation*}
\operatorname{Sing}^\times_{m,n}:= \bigcap_{c > 0}\UA_{m,n}^{\times}
(c\psi_1)
=\bigcap_{c>0}\left\{Y\in \mb{R}^{m\times n} \left|\begin{aligned} \eqref{eq:systemmult} \text{ has a non-trivial integer}\\\text{solution for all large enough }T \ \end{aligned}\right.\right\}
\end{equation*}
has full Lebesgue measure. Note that $\operatorname{Sing}^\times_{m,n}$
is a subset (and  a uniform analog) of
 the set $\operatorname{WA}^\times_{m,n}$ of matrices satisfying 
the Littlewood Conjecture. This motivates the following 
\begin{quest}
\label{quest:3} Is it true that (a) $\operatorname{Sing}^\times_{m,n}=\mb{R}^{m\times n} $? (b) $\operatorname{Sing}^\times_{m,n}=\operatorname{WA}^\times_{m,n}$?
\end{quest}

\begin{figure}[!h]
\centering

\tikzstyle{relation} = [text centered]

\tikzstyle{box} = [rectangle, minimum width=6.75cm, minimum height=3.5cm, text centered, draw=black!70]

\tikzstyle{boxx} = [rectangle, minimum width=6.75cm, minimum height=3cm, text centered, draw=black!70]

\begin{tikzpicture}

\node (box1) [boxx, align=center] {
\textbf{Well-Approximable Matrices} \\[2mm]
$\textup{WA}_{m,n}=\bigcap_{c>0}\textup{W}_{m,n}(c\psi_1)$, where \\[2mm]
$\textup{W}_{m,n}(\psi):=\limsup_{T\to\infty} \mc{S}_{m,n}\left(\psi,T\right)$; \\[2mm]
full Lebesgue measure.
};

\node [relation, right of=box1, xshift=2.75cm] {\scalebox{1.75}{$\supset$}};

\node (box2) [boxx, align=center, right of=box1, xshift=6.5cm] {
\textbf{Singular Matrices} \\[2mm]
$\textup{Sing}_{m,n}=\bigcap_{c>0}\textup{D}_{m,n}(c\psi_1)$, where \\[2mm]
$\textup{D}_{m,n}(\psi):=\liminf_{T\to\infty} \mc{S}_{m,n}\left(\psi,T\right)$; \\[2mm]
Hausdorff dim.\ $mn-\frac{mn}{m+n}$.
};

\node [relation, below of=box1, yshift=-1cm] {\scalebox{1.75}{\rotatebox{90}{$\supset$}}};

\node (box3) [box, align=center, below of=box1, yshift=-3.5cm] {
\textbf{Multiplicatively} \\
\textbf{Well-Approximable Matrices} \\[2mm]
$\textup{WA}_{m,n}^{\times}=\bigcap_{c>0}\textup{W}_{m,n}^{\times}(c\psi_1)$, where \\[2mm]
$\textup{W}^{\times}_{m,n}(\psi):=\limsup_{T\to\infty} \mc{S}^{\times}_{m,n}\left(\psi,T\right)$; \\[2mm]
conjecturally $\mb R^{m\times n}$ (Littlewood),\\
complement of Hausdorff dim.\ zero \\
(Einsiedler--Katok--Lindenstrauss).
};

\node [relation, below of=box2, yshift=-1cm] {\scalebox{1.75}{\rotatebox{90}{$\supset$}}};

\node (box4) [box, align=center, below of=box2, yshift=-3.5cm] {
\textbf{Multiplicatively} \\
\textbf{Singular Matrices} \\[2mm]
$\textup{Sing}_{m,n}^{\times}=\bigcap_{c>0}\textup{D}_{m,n}^{\times}(c\psi_1)$, where \\[2mm]
$\textup{D}^{\times}_{m,n}(\psi):=\liminf_{T\to\infty} \mc{S}^{\times}_{m,n}\left(\psi,T\right)$; \\[2mm]
full Lebesgue measure.
};

\node [relation, right of=box3, xshift=2.75cm] {\scalebox{1.75}{$\supset$}};

\end{tikzpicture}

\caption{Sets of interest in Diophantine approximation and their relation.}

\end{figure}
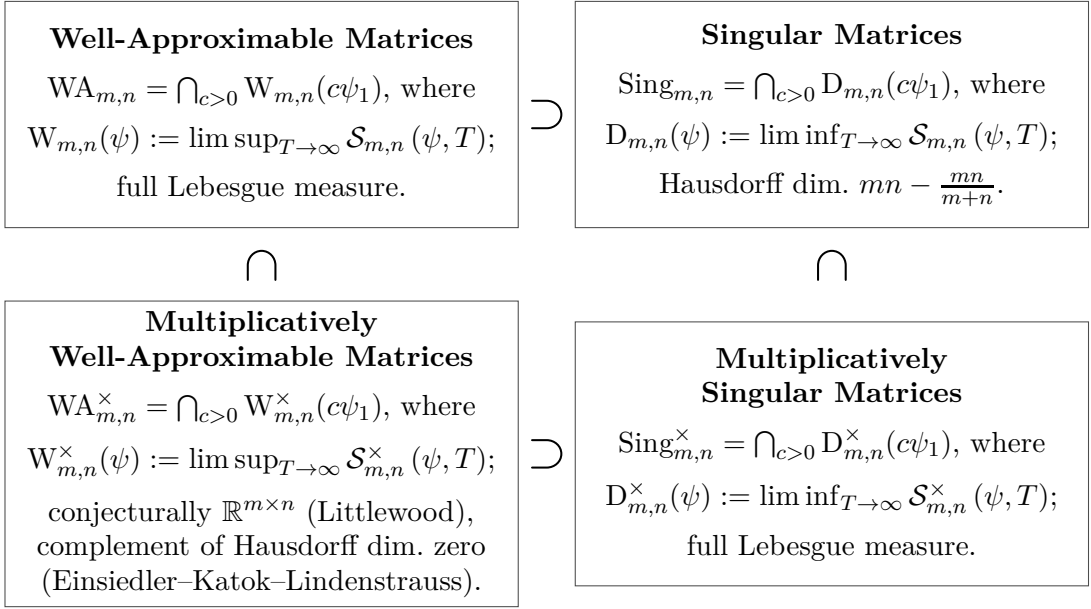

An affirmative answer to (a) would prove Littlewood's Conjecture in a (supposedly) stronger uniform form, while an affirmative answer to (b) would imply that the original and uniform versions of Littlewood's Conjecture are equivalent. It is quite likely though that, {at least in the case $\{m,n\} = \{1,2\}$,} the answers to both questions are negative; that is, there exist matrices $Y$ satisfying Littlewood's Conjecture but not its far-reaching uniform analog. 
Figure 1 summarizes all the various sets mentioned in the above discussion. The focus of this paper is on the lower-right box.\vspace{2mm}

\subsection{Outline of the paper}

Theorems \ref{thm:gmodgamma} and \ref{cor:diophantine}
{are} derived from an abstract result involving general measure-preserving actions, stated in Section \ref{sec:abstract}. The proof of this result can be found in Section \ref{sec:proofs}. 
{In Section \ref{sec:Diophcorr} we discuss in detail the correspondence between Diophantine approximation and dynamics on the space of lattices, and prove Theorem \ref{thm:special}.} The proof of Theorem \ref{thm:main} is presented in Section \ref{sec:Diophproof}, while Proposition \ref{prop:dno} is proved in Section~\ref{sec:dnoproof}. 

\noindent {\bf Acknowledgements.} The authors are grateful to Elon Lindenstrauss, Ralf Spatzier, and  Chengyang Wu for helpful discussions. 

\section{An Abstract Theorem and Proofs of Theorems \ref{thm:gmodgamma} and \ref{cor:diophantine}}
\label{sec:abstract}

\subsection{Two more results} Fix an integer $d\geq 2$ and let $(X,\mu)$ be a probability space equipped with a measure-preserving action $(a,x)\mapsto a.x$ of $\R^d$ on $X$. For any $\varphi\in\Lone(X,\mu)$ and any $a\in \mb R^d$ let us denote by $\varphi\circ a$ the function $\varphi\circ a(x):=\varphi(a.x)$, where $x\in X$. Let us also denote by $\|\cdot\|$ a fixed norm on $\mb R^d$. The following theorem will be used to derive Theorems \ref{thm:gmodgamma} and \ref{cor:diophantine}.

\begin{theorem}
\label{thm:genericgeneral} Let $\mathcal F \subset \textup{L}^{\infty}(X,\mu)$ be a sub-algebra of functions that is closed under the action of $\mb R^d$ and 
contains constant functions. Suppose that $\mathcal F$ is endowed with a norm $S$ with the following property: there exist constants $\rho,c>0$ such that for any $\phi,\psi\in \mc F$ and any $a\in \mb R^d$ 
one has that
\begin{equation}
\label{eq:norm1}
S(\phi\circ a)\leq S(\phi)\cdot e^{\rho \|a\|}    
\end{equation}
and
\begin{equation}
\label{eq:norm2}
S(\phi\cdot \psi)\leq c S(\phi)\cdot S(\psi).    
\end{equation}
Fix an integer $k$ with $1\le k \le d-1$ and let $M$ be a {compact} $k$-dimensional $\mathscr{C}^1$ submanifold
in $\mb R^d$. Let $\nu$ be a measure on $X$ with following property: there exists an integer $h\geq 2$ such that for any choice of $l$ functions $\varphi_1,\dots, \varphi_{l} \in \mc F$,with $1\leq l\leq h$, {any $t>0$}, and any choice of $l$ elements $a_1,\dots,a_{l}\in \{t\cdot g: g\in M\}$, the inequality
\begin{equation}
\label{eq:4mix}
\left|\nu\big((\varphi_1\circ a_1)\dotsm (\varphi_{l}\circ a_{l})\big)-\mu(\varphi_1)\dots \mu(\varphi_l)\right|\leq c'S(\varphi_1)\dotsm S(\varphi_{l})\cdot e^{-\sigma \Delta(a_1,\dotsc,a_{l})}\vspace{1.5mm}    
\end{equation}
holds, where  $\sigma,c'>0$ are fixed constants (potentially depending on $h$ and $d$) and 
$$\Delta(a_1,\dotsc,a_{l}):=\min\{\|a_i-a_j\|:i\neq j,\ i,j=1,\dotsc,l\}.$$
Further, let $(t_n)_{n\geq 1}$ be a sequence of positive real numbers satisfying
\begin{equation}\label{tn}
    t_{n+1}-t_n\geq \zeta>0
\end{equation}
for some constant $\zeta>0$. Then for any $\delta>0$ there exist a set $Y\subset X$, with $\nu(Y)=1$, and an integer $N$ (which may depend on the sequence $t_n$, the function $\varphi$, and $\delta$) such that for every $y\in Y$ and $n\geq N$ one has that
\begin{equation*}
\left|\frac1{t^{k}_n\vol_k(M)}\int_{t_nM}\varphi(a.y)\,\dd\vol_k(a)-\int_X\varphi\,\dd\mu\right|\leq t_n^{-k/2+1/h+\delta}.
\end{equation*}
\end{theorem}
Under the assumption that $X$ is a metric space and that the functions in $\mc F$ are regular enough, we may also obtain a continuous version of Theorem \ref{thm:genericgeneral}, which we now state.

\begin{theorem}
\label{thm:2}
Assume that the measure space {$(X,\nu)$} is endowed with a distance function $\dd$ and that all functions in the family $\mc F$ are Lipschitz with respect to such distance. Assume also that the action of $\R^d$ on $X$ is uniformly Lipschitz with respect to $\dd$. In other words, assume that there exists a constant $C>0$ such that for any $a_1,a_2\in \mb R^d$ and any $x\in X$
\begin{equation}
\label{eq:Lip}
\dd(a_1.x,a_2.x)\leq C\|a_1-a_2\|.    
\end{equation}
Then, under Assumptions {\eqref{eq:norm1}--\eqref{eq:4mix}}
 of Theorem \ref{thm:genericgeneral}, we have that the measure $\nu$ is {$(k,\mc F)$-generic for any $k = 1,\dots,d-1$}. Moreover, {for any {compact} $\mathscr{C}^1$ submanifold  
 $M\subset \mb R^d$ of dimension $1\leq k\leq d-1$} the rate of convergence in \eqref{eq:Birk-S} is given by $$O\left(t^{\left(-\frac{kh}{2}+\delta h+1\right)\frac{1}{1+h}}\right)$$
for any $\delta>0.$
\end{theorem}

\subsection{Proof of Theorem \ref{thm:gmodgamma} assuming Theorem \ref{thm:2}}
Let us begin by recalling the following result.

\begin{theorem}\cite[Theorem 1.1]{BEG20}
    \label{BEG}
    Let $G,\Gamma, X, \text{ and } \mu$ be as in Theorem \ref{thm:gmodgamma}.\\
    Then for every $h\geq 2$ and all sufficiently large $d\in \N$ there exists a constant $\sigma=\sigma(h,d)>0$ such that, for any $h$ functions $\varphi_1,\dots, \varphi_h \in \mathscr{C}_{c}^{\infty}(X)$ and any $h$ elements $g_1,\dots,g_h \in G$, we have that
    \begin{multline*}
        \left|\mu{\big((\varphi_1 \circ g_1) \dotsm (\varphi_h \circ g_h)\big)}-\mu(\varphi_1)\dotsm \mu(\varphi_h)\right| \\
        \ll_{h,d} S_d(\varphi_1)\dotsm S_d(\varphi_h) e^{-\sigma \min \{\dist(g_i,g_j)\,: \,1\leq i\neq j\leq h\}},
    \end{multline*}
    where $S_d$ denotes the $(2,d)$-Sobolev norm on $X$, and $\operatorname{dist}$ is a fixed left-$G$-invariant metric on $G$. {Here $\varphi \circ g$ denotes the function on $X$ defined by $(\varphi \circ g)(x):=\varphi(g.x)$}
    \end{theorem}

\begin{proof}[Proof of Theorem \ref{thm:gmodgamma}] With 
{Theorem \ref{BEG} at hand, the proof} 
follows readily from Theorem \ref{thm:2} with $\nu=\mu$ (the {$G$-invariant} probability measure on $G/\Gamma$). Note that the action of $G$ on $G/\Gamma$ is $1$-Lipschitz with respect to the induced distance function, and that the mixing assumption \eqref{eq:4mix} for functions in $\mathcal{F}=\mathscr{C}^{\infty}_c(X)$ and $h\geq 2$ is ensured by Theorem \ref{BEG}. The norm $S$ here is chosen to be the $(2,d)$-Sobolev norm $S_d$. Assumptions \eqref{eq:norm1} and \eqref{eq:norm2} are easily verified for this norm {and the conclusion follows.}
\end{proof}

\subsection{Proof of Theorem \ref{cor:diophantine} assuming Theorem \ref{thm:2}}
\label{subsec:2.2imp1.4}

{In order to prove Theorem \ref{cor:diophantine}, we require the following result on effective multiple mixing with respect to the measure $\nu$ introduced in Theorem \ref{cor:diophantine}. Note that $\nu$ is a Wiener measure
as defined in \cite{BG23}, hence the following theorem holds for $\nu$.}
\begin{theorem}\cite[Theorem 1.1]{BG23}
    \label{BG23}
    Let $G, \Gamma, X, \text{and } \nu$ be as in Theorem \ref{cor:diophantine}, {and let $E$ be as in \eqref{eq:E}}. Then
    for every $h\geq 1$ there exist $d=d(h)\in \N$ and $\sigma=\sigma(h)>0$ such that for any $h$ functions $\varphi_1,\dots, \varphi_h \in \mathscr{C}_c^{\infty}(X)$ and any $h$ elements $a_1,\dots,a_h\in E$ 
    we have that
    \begin{multline*}
        \left|\nu\Big(\varphi_1 \circ \exp(a_1) \dotsm \varphi_h \circ \exp(a_h)\Big)-\mu(\varphi_1)\dotsm \mu(\varphi_h)\right| \\
        \ll_{h,\nu} S_d(\varphi_1)\dotsm S_d(\varphi_h)\cdot e^{-\sigma \overline{\Delta}(a_1,\dots,a_h)},
    \end{multline*}
    where 
    $\overline{\Delta}(a_1,\dots,a_h):=\min \{\|a_i-a_j\|, \lfloor{a_i\rfloor}\,: \,1\leq i\neq j\leq h\}$. Here, $\|\cdot\|$ stands for a fixed norm on $\mf a\cong \mb R^{m+n-1}$, and $\lfloor a\rfloor$ denotes the quantity $\min_{i}|a_i|$ for any $a\in
    {E}$. The norm $S_d$ is once again assumed to be the $(2,d)$-Sobolev norm on the space $X$.
\end{theorem}

\begin{proof}[Proof of Theorem \ref{cor:diophantine}]
{Again we apply} Theorem \ref{thm:2}, {this time} with $\nu$ {as in Theorem \ref{cor:diophantine}}.
In this case, the mixing assumptions \eqref{eq:4mix} for functions in $\mathcal{F}=\mathscr{C}^{\infty}_c(X)$ and $h\geq 2$ follow from Theorem \ref{BG23}. However, a priori, the error term in Theorem \ref{BG23} depends on the function $\bar\Delta\leq \Delta$. To conclude, it suffices to show that there exists an absolute constant $\sigma'>0$ {(depending only on $M$)} such that for any $t>0$ and any choice of $h$ elements $a_1,\dotsc,a_h\in M\subset E$ it holds that
\begin{equation}
\label{eq:deltadeltabar}
\Delta(ta_1,\dotsc,ta_h)\leq \sigma' \bar\Delta(ta_1,\dotsc, ta_h).    
\end{equation}
To see this, observe that for any $t>0$ and any $a,b\in M$ we have that
$$\frac{\|ta-tb\|}{\lfloor ta\rfloor}=\frac{\|a-b\|}{\lfloor a\rfloor}.$$
Now, recall that $M$ {is assumed to be} compact. Therefore the function
$$(a,b)\mapsto \frac{\|a-b\|}{\lfloor a\rfloor}$$
is well-defined for any $a,b\in M$ and must admit a maximum
${m_0}$. 
Thus 
for any \linebreak $a,b\in M$ and any $t>0$ it holds that $\|ta-tb\|\leq {m_0}\lfloor ta\rfloor$. Hence \eqref{eq:deltadeltabar} holds with $\sigma'=\max\{1,{m_0}\}$.
\end{proof}

\section{Proof of Theorems \ref{thm:genericgeneral} and \ref{thm:2}}
\label{sec:proofs}

To prove Theorem \ref{thm:genericgeneral} we begin by estimating the $h$-th moment of the Birkhoff averages via effective mixing. With this estimate at hand, Theorem \ref{thm:genericgeneral} is obtained as a consequence of the Borel–Cantelli lemma. We now establish the required moment estimate. Let us denote by $H_k$ the set
$$H_{k}:=\{(s_1,\dotsc,s_k)\in\mb R^k:s_k\geq 0\}.$$

\begin{lem}
\label{lem:4mom}
Let $J\subset \mb R^{k}$ be a {bounded open subset of the upper-half space $H_k$}, and let $\gamma:J\to {M}$ be the parametrization of a $k$-dimensional manifold of class $\mathscr{C}^1$ such that 
\begin{equation}
\label{eq:lambdas}
\lambda_{\sup}:=\sup_{\substack{\bs s,\bs s'\in J \\ \bs s\neq\bs s'}}\frac{\|\gamma(\bs s)-\gamma(\bs s')\|}{\|\bs s-\bs s'\|}<\infty\quad\mbox{and}\quad\lambda_{\inf}:=\inf_{\substack{\bs s,\bs s'\in J \\ \bs s\neq\bs s'}}\frac{\|\gamma(\bs s)-\gamma(\bs s')\|}{\|\bs s-\bs s'\|}>0.
\end{equation}
Let $f:J\to (0,\infty)$ be a bounded Borel function, and let $\mc F$, $S$, and $\nu$ be as in Theorem \ref{thm:genericgeneral}, {with Assumptions \eqref{eq:norm1}--\eqref{eq:4mix} satisfied}. 
Then there exists a constant $\eta>0$ (depending on $h,\rho,\sigma$, and the parametrization $\gamma$) such that for any function $\varphi\in \mc F$ with $\mu(\varphi)=0$ and any $0<{b}<1$ we have that
{\begin{multline*}
\int_{X}\left(\int_{J}\varphi\left(t\gamma(\bs s).x\right)f(\bs s)\,\dd \bs s\right)^h\,\dd\nu(x)
 \\
=\|f\|_{\infty}^h\cdot O\left(A_{\varphi,h}^{\lfloor h/2\rfloor}\cdot {b}^{k(h-\lfloor h/2\rfloor)}+B_{\varphi,h}\cdot {b}^{k(h-1)}+C_{\varphi,h}\cdot e^{-{b} t}\right),
\end{multline*}
where
$$A_{\varphi,h}:=\max\left\{\mu(|\varphi|),\mu(|\varphi|^h)\right\},\quad B_{\varphi,h}:=\max\left\{\|\varphi\|^h_{\infty},A_{\varphi,h}\right\},$$
and $C_{\varphi,h}:=S(\varphi)^h$.}
\end{lem}

{We remark that the norms $\|\cdot\|$ on $\R^k$ and $\R^d$ appearing in \eqref{eq:lambdas} can be arbitrary: it is clear that the validity of \eqref{eq:lambdas} does not depend on the choice of norms.} The following lemma shows that the hypotheses in Lemma \ref{lem:4mom} are always satisfied for local parametrizations $\gamma$ of $\mathscr{C}^1$ submanifolds.

\begin{lem}
\label{lem:lambdainfsup}
{Let $J$ be a bounded convex open subset of the upper-half space $H_k$, and let $\gamma: J\to M$ be the restriction of a $\mathscr{C}^1$ map defined on a neighborhood of the set $\bar J$ in $H_k$, which has the form
$$\gamma(\bs s)=\big(\bs s,\gamma_{k+1}(\bs s),\dotsc,\gamma_{d}(\bs s)\big)$$
for $\bs s\in J$, perhaps after a permutation of the coordinates. Then the quantities $\lambda_{\sup}$ and $\lambda_{\inf}$ 
in \eqref{eq:lambdas} are positive and finite for the map $\gamma$.}
\end{lem}

\begin{proof}
 Without loss of generality, let us assume that
$$\gamma(\bs s)=\big(\bs s,\gamma_{k+1}(\bs s),\dotsc,\gamma_{d}(\bs s)\big)$$
for $\bs s\in J$. Then, {if  the norms $\|\cdot \|$ on $\mb R^k$ and $\mb R^d$ are both supremum norms}, it is obvious that
$$\|\gamma(\bs s)-\gamma(\bs s')\|\geq \|\bs s-\bs s'\|$$
for any $\bs s,\bs s'\in J$. This shows that $\lambda_{\inf}\geq 1$.

Now, denote by $\gamma_j$ for $j=1,\dotsc,d$ the components of $\gamma$. Then, by the Mean Value Theorem applied to the function $t\mapsto\gamma\big(t\bs s+(1-t)\bs s'\big)$ for fixed $\bs s\neq\bs s'$ in $J$, we have that
$$\gamma(\bs s')-\gamma(\bs s)=\left(\sum_{i=1}^k\frac{\partial\gamma_j}{\partial s_i}(\theta_j\bs s-(1-\theta_j)\bs s')(s_i-s_i')\right)_{j=1,\dotsc,d}$$
where $\theta_1,\dotsc,\theta_d\in[0,1]$. This implies that
$$\left\|\gamma(\bs s)-\gamma(\bs s')\right\|\leq k\sup_{
\bs \xi\in \bar J,\, 
j=1,\dotsc,d}\|\nabla\gamma_j(\xi)\|\|\bs s-\bs s'\|,$$
showing that $\lambda_{\sup}$ is bounded.
\end{proof}

{\begin{rmk}\label{rem:reduction}
Since $M$ is a compact $\mathscr{C}^1$ submanifold, 
it admits a \emph{finite} atlas where each chart is of the form described in Lemma \ref{lem:lambdainfsup}. Thus, working locally, we may assume from now on that $M$ is globally parametrized by a map $\gamma:J\to M$ satisfying the hypotheses of Lemma \ref{lem:4mom}.     
\end{rmk}}

\subsection{A covering lemma}
To prove Lemma \ref{lem:4mom}, we use a method introduced by Bj\"orklund and Gorodnik which allows us to analyze multiple correlations. Let us start by introducing some notation.

Let $(X,d)$ be any metric space. Fix $h\in \mb N$ and let $I,I'$ be disjoint subsets of $\{1,\dotsc, h\}$. For $x=(x_1,\dots,x_h)\in X^h$ we define 
$$\rho^{I}(x):=\max \{d(x_i,x_j):i,j\in I\}\quad\mbox{and}\quad\rho_{I,I'}(x):=\min\{d(x_i,x_j):i\in I,j\in I'\}.$$
Moreover, for a given partition  $\mc Q$ of the set $\{1,\dotsc,h\}$ we set
$$\rho^{\mc Q}(x):=\max\left\{\rho^{I}(x):I\in \mc Q\right\}\quad\mbox{and}\quad\rho_{\mc Q}(x):=\min\left\{\rho_{I,I'}(x):I,I'\in \mc Q,\ I\neq I'\right\}.$$
Finally, for $\mc Q$ as above, with $\#\mc Q\geq 2$, and $0\leq \alpha<\beta$ we let
$$\Delta_{\mc Q}(\alpha,\beta):=\left\{x\in X^{h}:\rho^{\mc Q}(x)\leq \alpha\ \mbox{and }\rho_{\mc Q}(x)>\beta\right\}$$
and
$$\Delta(\alpha):=\left\{x\in X^{h}:\rho^{\{1,\dotsc, h\}}(x)\leq \alpha\right\}.$$
{We now state the required covering lemma \cite[Proposition 6.2]{BG20new}. While \cite[Proposition 6.2]{BG20new} is stated under the assumption that $X$ is a locally compact group, the argument presented there applies equally well to any metric space.}
\begin{lem}
\label{lem:covering}
Let $h\in \mb N$ and let $\alpha_l,\beta_l$ for $l=0,\dotsc,h+1$ be real numbers such that
$$0=\alpha_0<\beta_1<\alpha_1=3\beta_1<\beta_2<\dotsb<\beta_h<\alpha_h=3\beta_h<\beta_{h+1}.$$
Then
$$X^{h}=\Delta(\beta_{h+1})\cup\bigcup_{l=0}^h\bigcup_{\mc Q:\# \mc Q\geq 2}\Delta_{\mc Q}(\alpha_l,\beta_{l+1}).$$
\end{lem}



We will use Lemma \ref{lem:covering} to partition the set $J^h$, after possibly removing any overlaps. For the sake of simplicity, we will continue using the symbols $\Delta(\beta_{h+1})$ and $\Delta_{\mc Q}(\alpha_k,\beta_{k+1})$ to denote the smaller sets in this partition, but with the understanding that they do not overlap and are contained in $J^h$. Here, $J$ is equipped with the metric induced by the norm $\|\cdot\|$ on $\R^k$.

\subsection{Proof of Lemma \ref{lem:4mom}} 
By multiplying out the $h$-th power, we obtain that
\begin{equation*}
\int_{X}\left(\int_{J}\varphi\left(t\gamma(\bs s).x\right)f(\bs s)\,\dd \bs s\right)^h\,\dd\nu(x)=\int_{J^h}\int_{X}\prod_{i=1}^h\varphi\big(t\gamma(\bs s_i).x\big)\,\dd\nu(x) \prod_{i=1}^{h}f(\bs s_i)\,\dd\ul{\bs s},
\end{equation*}
where $\ul {\bs s}=(\bs s_1,\dotsc,\bs s_h)\in J^h$. We use the partition induced by Lemma \ref{lem:covering} to subdivide the integration domain $J^h$. In particular, for some choice of the constants $\alpha_l,\beta_l$ (which will be made in the sequel of the proof), we estimate separately the integrals
\begin{equation*}
\int_{\Delta_{\mc Q}(\alpha_l,\beta_{l+1})}\int_{X}\prod_{i=1}^h\varphi\big(t\gamma(\bs s_i).x\big)\,\dd\nu(x) \prod_{i=1}^{h}f(\bs s_i)\,\dd\ul{\bs s}
\end{equation*}
for $l=0,\dotsc,h$ and
\begin{equation}
\label{eq:center}
\int_{\Delta(\beta_{h+1})}\int_{X}\prod_{i=1}^h\varphi\big(t\gamma(\bs s_i).x\big)\,\dd\nu(x) \prod_{i=1}^{h}f(\bs s_i)\,\dd\ul{\bs s}.    
\end{equation}
For each $\emptyset\neq I\subseteq \{1,\dotsc,h\}$ let $i(I)\in I$ be a fixed index, and, for $x\in X$, put
$$\varphi_{I,t,\ul{\bs s}}(x):=\prod_{i\in I}\varphi\left(t\left(\gamma(\bs s_i)-\gamma\left(\bs s_{i(I)}\right)\right).x\right).$$
Then for any partition $\mc Q$ of the set $\{1,\dotsc,h\}$ with $\#\mc Q\geq 2$ and any $0\leq l\leq h$ we have that
\begin{multline}
\label{eq:grouping4}
\int_{\Delta_{\mc Q}(\alpha_l,\beta_{l+1})}\int_{X}\prod_{i=1}^{h}\varphi(t\gamma(\bs s_i).x)\,\dd\nu(x)\prod_{i=1}^{h}f(\bs s_i)\,\dd\ul{\bs s} \\
=\int_{\Delta_{\mc Q}(\alpha_l,\beta_{l+1})}\int_{X}\prod_{I\in\mc Q}\varphi_{I,t,\ul{\bs s}}\left(t\gamma_{\bs{s}_{i(I)}}.x\right)\,\dd\nu(x)\prod_{i=1}^{h}f(\bs s_i)\,\dd\ul{\bs s}.
\end{multline}

Let us start by analyzing \eqref{eq:grouping4}. If $\ul{\bs s}\in \Delta_{\mc Q}(\alpha_l,\beta_{l+1})$, it follows from \eqref{eq:norm1}, \eqref{eq:norm2}, and the definition of $\varphi_{I,t,\ul{\bs s}}$, that for any subset $I\in\mc Q$
$$S(\varphi_{I,t,\ul{\bs s}})\ll S(\varphi)^{\#I}\cdot \exp\left(\# I\rho \max_{i\in I }\left\|\gamma(\bs s_i)-\gamma\left(\bs s_{i(I)}\right)\right\|t\right)\leq S(\varphi)^{\# I}\cdot e^{\# I\rho\lambda_{\sup}\alpha_{l}t}.$$
We now choose the constants $\alpha_l$ and $\beta_l$ to satisfy the relations
$$\begin{cases}
2h\rho\lambda_{\sup}\cdot \alpha_{l}\leq \sigma\lambda_{\inf}\cdot \beta_{l+1} & \mbox{for } l=1,\dotsc,h  \\[2mm]
2b\leq \sigma\lambda_{\inf}\cdot\beta_1\leq \sigma\lambda_{\inf}\cdot \beta_{h+1}\ll {b}
\end{cases},$$
where $0<{b}<1$ is fixed.
In particular, we may set
\begin{equation*}
\begin{cases}
\beta_l:=\dfrac{2b}{\sigma\lambda_{\inf}}\left(\dfrac{6h\rho\lambda_{\sup}}{\sigma\lambda_{\inf}}\right)^{l-1} & \mbox{for }l=1,\dotsc,h+1\\[3mm]
\alpha_{l}:=3\beta_l & \mbox{for }l=1,\dotsc,h
\end{cases},    
\end{equation*}
and $\alpha_0:=0$, where we assume, without loss of generality that, $\sigma<1$ and $\rho>1$.

With this choice of the constants $\alpha_l$ and $\beta_l$, we have that
\begin{equation}
\label{alphabetadiff}
h\rho\lambda_{\sup}\cdot \alpha_{l}-\sigma\lambda_{\inf}\cdot \beta_{l+1}\leq -\sigma\lambda_{\inf}\cdot \beta_{l+1}/2    
\end{equation}
for $l=0,\dotsc,h$. Then, on applying \eqref{eq:4mix} and \eqref{alphabetadiff}, we deduce that for $\ul{\bs s}\in \Delta_{\mc Q}(\alpha_l,\beta_{l+1})$ it holds that
\begin{multline*}
\left|\int_{X}\prod_{I\in\mc Q}\varphi_{I,t,\ul{\bs s}}\left(t\gamma\left(\bs s_{i(I)}\right).x\right)\,\dd\nu(x)-\prod_{I\in\mc Q}\mu(\varphi_{I,t,\ul{\bs s}})\right| \\\ll_{c,c'} S(\varphi)^{h}\cdot e^{(h\rho\lambda_{\sup}\alpha_l-\sigma\lambda_{\inf}\beta_{l+1}) t}
\leq S(\varphi)^{h} \cdot e^{-\sigma\lambda_{\inf}\beta_{l+1} t/2} \leq S(\varphi)^{h}\cdot e^{-bt}, 
\end{multline*}
where $c$ and $c'$ are the constants appearing in Conditions \eqref{eq:norm2} and \eqref{eq:4mix}. In view of this, we conclude that
\begin{multline}
\label{eq:Qalphabeta}    
\int_{\Delta_{\mc Q}(\alpha_l,\beta_{l+1})}\int_{X}\prod_{i=1}^h\varphi\big(t\gamma(\bs s_i).x\big)\,\dd\nu(x) \prod_{i=1}^{h}f(\bs s_i)\,\dd\ul{\bs s} \\
= \int_{\Delta_{\mc Q}(\alpha_l,\beta_{l+1})}\prod_{I\in\mc Q}\mu\left(\varphi_{I,t,\ul{\bs s}}\right)\prod_{i=1}^hf(\bs s_i)\,\dd\ul{\bs s} +O\left(\|f\|_{\infty}^h S(\varphi)^h\cdot e^{-{b}t}\right).
\end{multline}

{Combining \eqref{eq:Qalphabeta} and \eqref{eq:center}, we obtain
\begin{align}
& \int_{X}\left(\int_{J}\varphi\left(t\gamma(\bs s).x\right)f(\bs s)\,\dd \bs s\right)^h\,\dd\nu\nonumber \\[2mm]
& =\sum_{\mc Q:\#\mc Q\geq 2}\sum_{l=0}^{h}\int_{\Delta_{\mc Q}(\alpha_l,\beta_{l+1})}\prod_{I\in\mc Q}\mu\left(\varphi_{I,t,\ul{\bs s}}\right)\prod_{i=1}^hf(\bs s_i)\,\dd\ul{\bs s}\nonumber \\[2mm]
& +\int_{\Delta(\beta_{h+1})}\int_{X}\prod_{i=1}^h\varphi\big(t\gamma(\bs s_i).x\big)\,\dd\nu(x)\prod_{i=1}^hf(\bs s_i)\,\dd\ul{\bs s}+O_h\left(\|f\|_{\infty}^hS(\varphi)^h\cdot  e^{-{b}t}\right).\label{eq:est}
\end{align}
Further, we note that, by the generalized H\"{o}lder Inequality, for any partition $\mc Q$ of $\{1,\dotsc,h\}$ and any subset $I$ of $\mc Q$ of cardinality at least $2$, it holds that
\begin{multline*}
  \mu\left(\varphi_{I,t,\ul{\bs s}}\right)\leq \int_{X} \prod_{i\in I}\left|\varphi\left(t\left(\gamma(\bs s_i)-\gamma\left(\bs s_{i(I)}\right)\right).x\right)\right|\dd \mu(x) \\
  \leq \prod_{i\in I}\left(\int_X\left|\varphi\left(t\left(\gamma(s_i)-\gamma\left(s_{i(I)}\right)\right).x\right)\right|^{\# I}\dd \mu(x)\right)^{\frac{1}{\# I}}
  =\mu \left(|\varphi|^{\#I}\right),
\end{multline*}
where we used invariance of the measure $\mu$ in the last equality.\\
Moreover, we make the crucial observation that whenever $\mc Q$ contains a singleton, it holds that
$$\prod_{I\in\mc Q}\mu\left(\varphi_{I,t,\ul{\bs s}}\right)=0,$$
since if $I$ is a singleton, $\mu\left(\varphi_{I,t,\ul{\bs s}}\right)=\mu(\varphi)=0$.
In view of this, the main term in \eqref{eq:est} is bounded above by
\begin{multline*}
\sum_{\substack{\mc Q:\#\mc Q\geq 2 \\ \# I\geq 2\ \forall\,I\in\mc Q}}\sum_{l=0}^{h}\operatorname{Leb}\big(\Delta_{\mc Q}(\alpha_l,\beta_{l+1})\big)\cdot \|f\|_{\infty}^h\prod_{I\in\mc Q}\mu\left(|\varphi|^{\# I}\right)+\operatorname{Leb}\big(\Delta(\beta_{h+1})\big)\cdot \|f\|_{\infty}^h\|\varphi\|_{\infty}^{h} \\
\ll_{h,\rho,\sigma,\lambda_{\sup},\lambda_{\inf}} \|f\|_{\infty}^h\sup_{\substack{\mc Q:\#\mc Q\geq 2 \\ \# I\geq 2\ \forall\,I\in\mc Q}}\prod_{I\in\mc Q}\mu\left(|\varphi|^{\# I}\right)\cdot {b}^{k(\# I-1)}+\|f\|_{\infty}^{h}\|\varphi\|_{\infty}^h\cdot {b}^{k(h-1)} \\[2mm]
\ll_h\|f\|_{\infty}^h\cdot \left(\sup_{\substack{\mc Q:\#\mc Q\geq 2 \\ \# I\geq 2\ \forall\,I\in\mc Q}}\left(\sup_{i=1}^h\mu\left(|\varphi|^i\right)\right)^{\# \mc Q}\cdot {b}^{k(h-\#\mc Q)}+\|\varphi\|_{\infty}^h\cdot {b}^{k(h-1)}\right) \\[2mm]
\ll_h \|f\|_{\infty}^h\cdot\left(A_{\varphi,h}^{\lfloor h/2\rfloor}\cdot {b}^{k(h-\lfloor h/2\rfloor)}+B_{\varphi,h}\cdot {b}^{k(h-1)}\right),  
\end{multline*}
where
$$A_{\varphi,h}:=\max\left\{\mu(|\varphi|),\mu(|\varphi|^h)\right\}\quad\text{and}\quad B_{\varphi,h}:=\max\left\{\|\varphi\|^h_{\infty},A_{\varphi,h}\right\}.$$
The upshot is that, for any $0<{b}<1$ one has that
\begin{multline*}
\int_{X}\left(\int_{J}\varphi\left(t\gamma(\bs s).x\right)f(\bs s)\,\dd\bs s\right)^h\,\dd\nu \\
=\|f\|_{\infty}^h\cdot O\left(A_{\varphi,h}^{\lfloor h/2\rfloor}\cdot {b}^{k(h-\lfloor h/2\rfloor)}+B_{\varphi,h}\cdot {b}^{k(h-1)}+S(\varphi)^h\cdot e^{-{b}t}\right),  
\end{multline*}
where the implicit constant depends on $c,c',h,\rho,\sigma,\lambda_{\sup},$ and $\lambda_{\inf}$. This concludes the proof.}

\subsection{Completion of the proof {of Theorem \ref{thm:genericgeneral}}}
Building on the previous bounds, we now establish the following result, which immediately implies Theorem \ref{thm:genericgeneral}.

\begin{lem}
\label{lem:BC1}
Let $J\subset\mb R^k$ be a {bounded open subset of the upper-half space $H_k$}, let $\gamma:J\to {M}$ be as in Lemma~\ref{lem:4mom}, and let {$f
$} be the density of the pull-back measure induced by $\vol_k$ through $\gamma$ on the set $J$, namely:
$$f(\bs s):=\sqrt{\det(\nabla\gamma(\bs s)\nabla\gamma(\bs s)^{T})}.$$
Suppose that $t_n$ is a sequence satisfying \eqref{tn} and that $\varphi\in\mc F$ is a function with $\mu(\varphi)=0$. Then for any $\delta>0$ and $\nu$-a.e. $x\in X$ there exists $N\in \mb N$ such that for $n\geq N$ it holds that
$$\left|\int_{J}\varphi\left(t_n\gamma(\bs s).x\right)f(\bs s)\,\dd \bs s\right|\leq t_n^{-k/2+1/h+\delta}.$$    
\end{lem}

\begin{proof}
For $n\in\mb N$ let
$$S_{n}:=\left\{x\in X:\left(\int_{J}\varphi\left(t_n\gamma(\bs s)\cdot x\right)f(\bs s)\,\dd \bs s\right)^h> t_n^{-kh/2+1+h\delta}\right\}.$$
On applying Markov's Inequality and Lemma \ref{lem:4mom}, with ${b}=t_n^{-1+\delta/k}$, we find that
\begin{equation*}\nu(S_n)\leq \frac{1}{t_n^{-kh/2+1+h\delta}}\int_{X}\left(\int_{J}\varphi\left(t_n\gamma(\bs s)\cdot x\right)f(\bs s)\,\dd \bs s\right)^h\,\dd\nu\ll_{k,h,\gamma,\varphi,\delta} t_n^{-1-h\delta/2}. 
\end{equation*}
Since $t_n$ satisfies \eqref{tn}, it follows that $\sum_n \nu(S_n)<\infty$. 
Hence, by the Borel--Cantelli Lemma, we have that for $\nu$-almost every $x\in X$ there exists $N\in \N$ (depending on $x$) such that for every $n\geq N$ it holds that
$$\left|\int_{J}\varphi\left(t_n\gamma(\bs s).x\right)f(\bs s)\,\dd \bs s\right|\leq t_n^{-k/2+1/h+\delta}.$$
\end{proof}

\subsection{Proof of Theorem \ref{thm:2}}

The proof is analogous to that of Theorem \ref{thm:genericgeneral}, with the only difference that the sequence $t_n$ is chosen to have decreasing spacing. Let $J,\gamma,$ and $f$ be as in the previous subsection. Let also $t_n:=n^{1-A}$, where $0<A<1$ is a constant yet to be chosen. Consider the sets
$$S_{n}:=\left\{x\in X:\left(\int_{J}\varphi\left(t_n\gamma(\bs s)\cdot x\right)f(\bs s)\,\dd \bs s\right)^h> t_n^{-\frac{kh}{2}+\frac{1}{1-A}+h\delta}\right\}$$
for $n\geq 1$. Then, following the proof of Lemma \ref{lem:BC1}, we obtain that
\begin{multline*}
\nu(S_n)<\frac{1}{t_n^{-\frac{kh}{2}+\frac{1}{1-A}+h\delta}}\int_X\left(\int_{J}\varphi\left(t_n\gamma(\bs s)\cdot x\right)f(\bs s)\,\dd \bs s\right)^h\,\dd\nu \\
\ll_{k,h,\gamma,\varphi,\delta} t_n^{-\frac{1}{1-A}-\frac{h\delta}{2}}=n^{-1-\frac{h\delta}{2}(1-A)}.
\end{multline*}
Hence the series $\sum_n\nu(S_n)$ converges.
By the Borel--Cantelli Lemma, we deduce that for $\nu$-a.e. $x\in X$, and for all large enough $n$ (depending on $x$)
\begin{equation*}
    \int_{J}\varphi\left(n^{(1-A)}\gamma(\bs s)\cdot x\right)f(\bs s)\,\dd \bs s \leq  n^{(1-A)\left(-\frac{k}{2}+\frac{1}{(1-A)h}+\delta\right)}. 
\end{equation*}
Then for any $t\in \left[n^{1-A},(n+1)^{1-A}\right)$ it holds that
\begin{align*}
\int_{J}\varphi\left(t\gamma(\bs s)\cdot x\right)f(\bs s)\,\dd \bs s & \ll_{A,J} \int_{J}\varphi\left((n+1)^{1-A}\gamma(\bs s)\cdot x\right)f(\bs s)\,\dd \bs s +\|\varphi\|_{Lip}Cn^{-A} \\
&\ll_{A,C} t^{-\frac{k}{2}+\frac{1}{(1-A)h}+\delta} + \|\varphi\|_{\operatorname{Lip}}t^{-\frac{A}{1-A}}, \\
\end{align*}
where we used the fact that both the function $\varphi$ and the action of $\mb R^d$ on $X$ are Lipschitz, and $C$ is the constant appearing in \eqref{eq:Lip}. We now choose the parameter $A$ so the exponents in the above two terms coincide. Specifically, we set 
$$A=\frac{k/2-\delta-\frac{1}{h}}{k/2-\delta+1}.$$
With this choice of $A$, we deduce that
$$\int_{J}\varphi\left(t\gamma(\bs s)\cdot x\right)f(\bs s)\,\dd \bs s \ll_{k,h,\gamma,\varphi,\delta,C} t^{(-\frac{kh}{2}+\delta h+1)\frac{1}{1+h}}.$$
By increasing $\delta$, we may assume that the implicit constant is $1$, thus \eqref{eq:Birk-S-eff} follows and the proof is completed.
\qed

\section{Correspondence with Dynamics 
and Proof of   Theorem \ref{thm:special}}
\label{sec:Diophcorr}

\subsection{The multiplicative Dani Correspondence}
In this section  we explain how to derive 
Theorem \ref{thm:special} from 
Theorem \ref{cor:diophantine},
by rephrasing the problem in a dynamical language. It is well-known that Diophantine properties of a given matrix $Y\in\mb R^{m\times n}$ can be  expressed
in terms of the long-term behavior of a certain orbit in the space of $(m+n)$-dimensional unimodular lattices (this dates back to the work of Dani \cite{Dani}). 
A correspondence with dynamics in the multiplicative set-up was more recently developed in \cite{FK24}. We briefly recall it here for the reader's convenience. We begin by presenting a lemma that underpins this whole construction.

\begin{lem}\rm (\cite[Lemma 4.1]{FK24}, a special case of \cite[Lemma 8.3]{KM99})
\label{lem:C1}  
For any $m,n\in\N$ 
and  any  {continuous}
 non-increasing function $\psi:{(}0,\infty)\to(0,1]$ there exists a unique 
{continuous}
 function 
 $R:[0,\infty)\mapsto \R$  
such that %
\eq{incr}{\text{the map }t\mapsto t-nR(t) \text{ is strictly increasing and tends to $\infty$ as } t\to\infty,}
\eq{nondecr}{\text{the map }t\mapsto t+mR(t)\text{ is non-decreasing,}}
and 
\eq{corr}{\psi\left(e^{t-nR(t)}\right)=e^{-t-mR(t)}\quad\forall\,t\geq 0.}
Conversely, given 
a 
 {continuous}
function $R:[0,\infty)\to \mb{R}$ with $R(0) \ge 0$, satisfying 
\equ{incr} 
and 
\equ{nondecr},  
there exists a unique 
 {continuous}
    non-increasing function $\psi:[e^{- nR(0)},\infty)\to(0,1]$
    such that \equ{corr} holds.
\end{lem}
\noindent To state the precise correspondence {between multiplicative approximation and dynamics on the space of lattices}, we first need to introduce some notation. Let $G=:\operatorname{SL}_{m+n}(\R)$ and 
$\Gamma:=\operatorname{SL}_{m+n}(\Z)$. Denote by $A$ the subgroup of diagonal matrices in $G$ and let $${\mf a:=\operatorname{Lie}(A) = \big\{ \operatorname{diag}(a_1,\dots,a_{m+n}): a_1+\dotsb+a_{m+n} = 0\big\}}$$ be the Lie algebra of $A$. Now, given a function $R:[{0},\infty)\to \R$ and 
${t \ge 0}$, define the set
\begin{multline*}
\label{eq:CRt}
C_{R}(t):=  
\big\{a\in\mf a : a_1+\dotsb+a_m={-(a_{m+1}+\dotsb+a_{m+n})}= t,\\
    a_{i}> -R(t)\mbox{ for }i=1,\dotsc,m, 
    \mbox{ and }{a_{m+j}<-R(t)}\mbox{ for }j=1,\dotsc,n
\big\},
\end{multline*}
and put
$$C_R:=\bigcup_{t> 0}C_R(t).$$
{Note that when $R\equiv 0$, the set $C_R = C_0$ coincides with $E$ as defined in \eqref{eq:E}.}

Recall that the homogeneous space $X:=G/\Gamma$ can be 
identified with the space of unimodular lattices in $\R^{m+n}$
{by letting} 
$g\Gamma \in X$ correspond to the unimodular lattice $g\Z^{m+n}$. Given a lattice ${x}\in X$, we denote by $\delta({x})$ the length of the shortest non-zero vector of ${x}$ with respect to the supremum norm, i.e.,
\begin{equation*}
    \delta({x}):=\min\big\{\|\bs v\|_{\infty}: \bs v\in {x}\setminus \{\bs 0\} \big\}.
\end{equation*} 
For $Y\in \R^{m\times n}$  we put
    \begin{equation*}
    {x}_{Y}:=\begin{pmatrix}
        I_{m}& Y\\
        \bs 0 & I_{n}
    \end{pmatrix}
    \Z^{m+n}.
\end{equation*}
{Note that $\{{x}_{Y}: Y\in \R^{m\times n}\}$ is a   compact orbit of $U_{m,n}$ as in \eqref{eq:umn}. Finally, as in Theorem \ref{thm:gmodgamma}, we consider the action $(a,x)\mapsto a.x$ of $\mathfrak a$ on $X$ given by \eqref{eq:action}.}

The following result was proved in \cite{FK24} and represents the core of the correspondence in the multiplicative set-up.

\begin{prop}
\label{prop:C2} \cite[Proposition  4.4]{FK24}
Let $\psi:{(0},\infty)\to(0,1]$ be a  {continuous}
non-increasing function, and let $R$ be the function corresponding to $\psi$ via Lemma  \ref{lem:C1}.
Take $Y\in \mb{R}^{m\times n}$ and $T\geq 1$. Then 
$Y\in\mc{S}_{m,n}^{\times}(
\psi,T)$ if and only if 
{$\exists\,a\in C_{R}(t)$, with $t\ge 0$ defined by  $T=e^{t-nR(t)}$,} such that
$\delta ({a.x}_{Y}) < e^{-R(t)}$.
\end{prop} 

From Proposition \ref{prop:C2}  one can {immediately} deduce a 
characterization of {sets $\operatorname{W}_{m,n}^{\times}(\psi)$ and $\UA _{m,n}^{\times}(\psi)$ (see \equ{psima} and \equ{psidir})} in terms of the long-term behavior of the orbits ${\{a.x_{Y}: a\in C_R\}}$ in 
$X$. 

\begin{cor}
\label{prop:corr}
Let $\psi:{(0},\infty)\to(0,1]$ be a  {continuous}
non-increasing function, and let $R$ be the function corresponding to $\psi$ via Lemma \ref{lem:C1}. {Then:
\begin{itemize}
    \item[{\rm (a)}] $Y\in \UA ^{\times}_{m,n}(\psi)$ if and only if there exists $t_0> 0$ such that whenever $t\geq t_0$ 
\begin{equation}
\label{eq:hyper}
\exists\, a\in C_R(t):\delta ({a.x}_{Y}) < e^{-R(t)};
\end{equation}
   \item[{\rm (b)}] (see \cite[Proposition {4.5}]{FK24}) $Y\in \operatorname{W} ^{\times}_{m,n}(\psi)$ 
   if and only if there exists an unbounded set of $t\ge 0$ such that \eqref{eq:hyper} holds.
\end{itemize}}
\end{cor}

\subsection{A dynamical description of $\operatorname{Sing}_{m,n}^{\times}$ and $\operatorname{WA} _{m,n}^{\times}$} As an example, fix $0 < c \leq 1$ and let ${\psi=\min\{1,c\psi_1\}}$; then the function $R:[0,\infty)\mapsto \R$ corresponding to $\psi$ through Lemma \ref{lem:C1} is constant, namely
\begin{equation*}
\label{eq:R}
R(t) = R_c\text{ for all }t\ge 0, \text{ where }R_c  := \frac{1}{m+n}\log (1/c).    
\end{equation*}
Now, consider the set 
$C_{R_c} = \bigcup_{t>0}C_{R_c}(t),$ where 
\begin{multline*}
\label{eq:Cct}
C_{R_c}(t) =\left\{a\in \mf a:a_1+\dots+a_m={-(a_{m+1}+\dotsb+a_{m+n})}=t, \right.\\ \left.a_i> \frac{\log c}{m+n}\mbox{ for }i=1,\dotsc,m,
\mbox{ and }a_{m+j}<\frac{\log c}{m+n}\mbox{ for }j=1,\dotsc,n\right\}.
\end{multline*}
In view of Corollary \ref{prop:corr}, this cone appears in the dynamical description of the sets $\UA ^{\times}_{m,n}(c\psi_1)$ and  $\operatorname{W} ^{\times}_{m,n}(c\psi_1)$. Clearly $R_1 = 0$, so that $C_{R_1}$ coincides with the cone $C_0 = E$ as in \eqref{eq:E}. Furthermore, for $c<1$ the  sets $C_{R_c}$ are at a bounded Hausdorff distance from $E$. However, since $c$ is changing in the definitions of the sets $\textup{WA}_{m,n}^{\times}$ and $\textup{Sing}_{m,n}^{\times}$, one may not use the cones $C_{R_c}$ to dynamically describe these sets.

Our next observation is that, by a slight perturbation of the constant $c$, one can characterize both the sets $\UA ^{\times}_{m,n}(c\psi_1)$ and  $\operatorname{W} ^{\times}_{m,n}(c\psi_1)$ solely via the action of the cone $E$ on $X$. 
This is formalized in the following "Transference Lemma", which is a generalization of the argument from 
the proof of \cite[Proposition 11.1]{EKL06} to the matrix set-up.

\begin{lem}[Transference Lemma]
\label{lem:transference}
Let $Y\in\mb R^{m\times n}$, let $0<c<1$, and let $t>0$.
\begin{itemize}
    \item[{\rm (a)}] If $\delta(a.x_{Y})<c^{\frac{1}{m+n}}$ for some $a\in C_{R_c}(t)$, then  $$\exists\text{ $a'\in C_0(t)$}:\delta( a'.x_Y)<c^{\frac{1}{m(m+n)}}.$$  
   \item[{\rm (b)}] If $\delta(a.x_{Y})<{c}^{\frac{m+n-1}{m(m+n)}}$ for some  $a\in C_0(t)$, then
   $$\exists\text{ $a'\in C_{R_{c}}\left(t-\tfrac{n-1}{m+n} \log {c}\right)$}:\delta(a'.x_Y)<{c}^{\frac{1}{m+n}}.$$   
\end{itemize}
\end{lem}

\begin{proof}
Let us start by proving (a). Fix $a\in C_{R_c}(t)$ and $\bs p\in\mb Z^{m},\,\bs q\in\mb Z^{n}$ (not both $\bs 0$) such that
\begin{equation}
\label{eq:condp}
e^{a_i}|Y_i\bs q+p_i|<c^{\frac{1}{m+n}}\quad\mbox{for }i=1,\dotsc,m    
\end{equation}
and
\begin{equation}
\label{eq:condq}
e^{a_{m+j}}|q_j|<c^{\frac{1}{m+n}}\quad\mbox{for }j=1,\dotsc,n.
\end{equation}
{Since $a_1+\dots+a_m=t>0$, at least one of the $a_i$'s with $i=1,\dotsc,m$ should be positive.} Assume now, without loss of generality, that $a_1,\dotsc,a_k\leq 0$ and that $a_{k+1},\dotsc,a_m>0$ for some $1\leq k\leq m-1$. Let $\varepsilon>0$ to be determined later. By using Dirichlet's Theorem {in $\R^k$}, we may pick an integer $1\leq \ell\leq \varepsilon^{-1}$ such that
$$|Y_i\ell\bs q+p_i'|<\varepsilon^{1/k}\quad\mbox{for }i=1,\dotsc,k,$$
for some $p_i'\in\mb Z$. Then, by \eqref{eq:condp} and \eqref{eq:condq}, we deduce that
\begin{equation}
\label{eq:condDir}
\begin{cases}
|Y_i\ell\bs q+p_i'|<\varepsilon^{1/k}\quad\mbox{for }i=1,\dotsc,k \\[2mm]
e^{a_i}|Y_i\ell\bs q+\ell p_i|<\varepsilon^{-1}c^{\frac{1}{m+n}}\quad\mbox{for }i=k+1,\dotsc,m \\[2mm]
e^{a_{m+j}}|\ell q_j|<\varepsilon^{-1}c^{\frac{1}{m+n}}\quad\mbox{for }j=1,\dotsc,n.
\end{cases}
\end{equation}
{Since $a_1+\dots+a_m=t$ and $a_1,\dots,a_k\leq 0$, we have $a_{k+1}+\dots+a_m\geq t$.} Now choose $0\leq a_{i}'\leq a_{i}$ for $i=k+1,\dotsc,m$ so that
$$a_{k+1}'+\dotsb +a_{m}'=t,$$
and define
$$a':=(0,\dotsc,0,a_{k+1}',\dotsc,a_{m}',a_{m+1},\dotsc a_{m+n})\in \overline{C_0(t)}.$$
{Observe that $\bs q\neq 0$, since otherwise \eqref{eq:condp} implies that there exists $1\leq i\leq m$ such that $a_i<\frac{\log c}{m+n}$, which cannot happen since $a\in C_{R_c}(t)$.}
Thus \eqref{eq:condDir} implies that
\begin{equation}
\label{eq:opencond}
\delta(a'.x_Y)<\max\left\{\varepsilon^{1/k},\varepsilon^{-1}c^{\frac{1}{m+n}}\right\}. \end{equation}
Optimizing, we find that the best possible value for $\varepsilon$ is
$\varepsilon=c^{\frac{k}{(k+1)(m+n)}}$. 
For this $\varepsilon$, the estimate \eqref{eq:opencond}, gives
\begin{equation}
\label{eq:opencond1}
\delta(a'.x_Y)< c^{\frac{1}{(k+1)(m+n)}}\leq c^{\frac{1}{m(m+n)}}.
\end{equation}
To conclude the proof of (a), we observe that, since \eqref{eq:opencond1} is an open condition, by perturbing $a'$ slightly we may assume that it does not lie on the boundary of 
$C_0(t)$.

Let us now move on to the proof of (b). Fix $a\in C_0(t)$ and $\bs p\in\mb Z^{m},\,\bs q\in\mb Z^{n}$ (not both $\bs 0$) such that 
\begin{equation}
\label{eq:condp2}
e^{a_i}|Y_i\bs q+p_i|<{c}^{\frac{m+n-1}{m(m+n)}}\quad\mbox{for }i=1,\dotsc,m    
\end{equation}
and
\begin{equation}
\label{eq:condq2}
e^{a_{m+j}}|q_j|<{c}^{\frac{m+n-1}{m(m+n)}}\quad\mbox{for }j=1,\dotsc,n.
\end{equation}
As before, since $a_i > 0$ for all $i=1,\dots,m$, \eqref{eq:condp2} implies that $\bs q\ne\bs 0$, hence \eq{atleast}{e^{a_{m+j}}<{c}^{\frac{m+n-1}{m(m+n)}}\text{ for at least one }j = 1,\dots,n.}
Now, assume without loss of generality that  for some $1\leq k\leq n-1$ we have that 
$$e^{a_{m+j}}\geq c^{\frac{1}{m+n}}\text{ for $j=1,\dotsc,k$, \quad while \quad $e^{a_{m+j}}<c^{\frac{1}{m+n}}$ for }j=k+1,\dotsc,n.$$ Note that  we cannot have $k=n$ in view of \equ{atleast}. Let us set
$$a_{i}':=a_i-\frac{n-1}{m(m+n)}\log c\ \mbox{ for }i=1,\dotsc,m, \text{ and $a_{m+j}':=\frac{1}{m+n}\log c$\  for }j=1,\dotsc,k.$$
Finally, let us choose $a_{m+j}'\leq a_{m+j}$ for $j=k+1,\dotsc,n$ in such a way that
$$-a_{m+1}'-\dotsb -a_{m+n}'=a_1'+\dotsb +a_m'=t-\frac{n-1}{m+n}\log c.$$
Then  $a'\in \overline{C_{R_{c}}\left(t-\frac{n-1}{m+n}\log c\right)}$. Moreover, \eqref{eq:condp2} and \eqref{eq:condq2} imply that
$$
\begin{aligned}e^{a_i'}|Y_i\bs q+p_i| &= c^{-\frac{n-1}{m(m+n)}}e^{a_i}|Y_i\bs q+p_i|\\ &<c^{-\frac{n-1}{m(m+n)}}{c}^{\frac{m+n-1}{m(m+n)}}= {c^{\frac{1}{m+n}}}\quad\mbox{for }i=1,\dotsc,m
\end{aligned}$$
and 
$$e^{a_{m+j}'}|q_j|\le e^{a_{m+j}}|q_j|<{c^{\frac{1}{m+n}}}\quad\mbox{for }j=1,\dotsc,n,$$
whence
\begin{equation*}
\delta(a'.x_Y)< c^{\frac{1}{(m+n)}}.
\end{equation*}
Once again, since 
both inequalities above are open conditions, we may perturb $a'$ slightly so that it  {lies 
inside 
${C_{R_{c}}\left(t-\frac{n-1}{m+n}\log{c}\right)}$.}
\end{proof}
Now, we can describe multiplicatively singular and multiplicatively well-approximable matrices using only the action of the cone $E$ on $X$. Let us define the following section of $E$: 
\eq{E1}{E_1 := C_0(1) = \{a\in E: a_1 + \cdots+a_m = 1\}.}
\begin{cor}
\label{cor:sing}
Let $Y\in\mb R^{m\times n}$. Then: 
\begin{itemize}
    \item[{\rm (a)}] $Y\in\operatorname{Sing}_{m,n}^{\times}$ 
    $\Longleftrightarrow$ for any $\varepsilon>0$ there exists $t_0>0$ such that,  whenever $t\geq t_0$,
    \begin{equation}
\label{eq:lessthanepsilon}
\begin{aligned}
\exists\, a\in C_0(t):\delta ({a.x}_{Y}) < \varepsilon,\\
\text{or, equivalently,}\ &(tE_1).x_{Y}\cap\{x\in X: \delta(x)<\varepsilon\}\neq\emptyset.
\end{aligned}
\end{equation}
   \item[{\rm (b)}]  $Y\in \operatorname{WA} ^{\times}_{m,n}$ $\Longleftrightarrow$ 
   for any $\varepsilon>0$ there exists an unbounded set of $t\ge 0$ such that \eqref{eq:lessthanepsilon} holds, or, equivalently, the orbit $E.x_Y$ is unbounded in $X$.
\end{itemize}
\end{cor}
\begin{proof}
This corollary is a simple consequence of Lemma \ref{lem:transference}  and Corollary \ref{prop:corr}. We explain the proof of  (a), the other part follows in a similar manner. 

Take $Y\in\operatorname{Sing}_{m,n}^{\times}$, choose an arbitrary $\varepsilon > 0$ and let $c=\varepsilon^{m(m+n)}$. Then we have $Y\in \UA_{m,n}^{\times}
(c\psi_1)$, hence, by Corollary \ref{prop:corr},
for all sufficiently large $t$ there exists $a\in C_{R_c}(t)$ with $$\delta ({a.x}_{Y}) < e^{-R_c} = c^{\frac{1}{m+n}}.$$ Lemma \ref{lem:transference} then produces $a'\in C_0(t)$ with $\delta( a'.x_Y)<c^{\frac{1}{m(m+n)}} = \varepsilon.$ 

Conversely,   choose an arbitrary $c > 0$ and let $\varepsilon = c^{\frac{m+n-1}{m(m+n)}}$.
If we know that for  all sufficiently large $t$ there exists $a\in C_0(t)$   such that $\delta ({a.x}_{Y}) < \varepsilon$, then Lemma \ref{lem:transference} implies that for  all sufficiently large $t'$ there exists $a'\in C_{R_{c}}(t')$ with $\delta(a'.x_Y)<{c}^{\frac{1}{m+n}}$. This,  in view of Corollary \ref{prop:corr}, shows that $Y\in \UA_{m,n}^{\times}
(c\psi_1)$. \end{proof}

We remark that part (b) generalizes \cite[Proposition 11.1]{EKL06} where the case $m=2$, $n=1$ was considered. The general case was stated without proof in \cite[Remark~4.5]{FK24}.

Another remark is that Corollary \ref{cor:sing}(a) and Mahler's Compactness Criterion can be used to equivalently restate Question \ref{quest:3} in the language of dynamics on the space of lattices as follows: 
\begin{quest}Is it true that \begin{itemize}
    \item[{\rm (a)}] there \textit{does not exist} $Y\in \R^{m\times n}$ such that 
    \eq{ULCexception}{(tE_1).x_{Y}\subset K\text{ for some compact set $K\subset X$ and an unbounded set of } t\ge 0?}
     \item[{\rm (b)}] if $Y$ satisfies \equ{ULCexception}, then the orbit $E.x_{Y}$ is bounded in $X$?
\end{itemize}
\end{quest}

We conclude this section by deducing Theorem \ref{thm:special} from Corollary \ref{cor:sing} and Theorem~\ref{cor:diophantine}.

\subsection{Proof of Theorem \ref{thm:special}}
In what follows, we make use of the notation introduced to state Theorem \ref{cor:diophantine}. Let $\nu$ be the $U_{m,n}$-invariant probability measure supported on the compact $U_{m,n}$-orbit $\{x_Y: Y\in \R^{m\times n}\}$. For a fixed $0<\sigma<\frac1{\max\{m,n\}}$ consider the following subset of $E\subset \mf a$:
\begin{equation}
\label{eq:Esigma}
M_{\sigma}:=\left\{a\in E_1:a_1,\dotsc,a_m\geq\sigma\mbox{ and } a_{m+1},\dotsc,a_{m+n}\leq-\sigma \right\},    
\end{equation}
where $E_1$ is as in \equ{E1}. Then $M_{\sigma} \subset E_1$ is a bounded $(m+n-2)$-dimensional $\mathscr{C}^1$ submanifold compactly contained in the cone $E=C_0$. 
By the first part of Theorem \ref{cor:diophantine}, we have that 
$\nu$ is $\big(M_{\sigma},\mathscr{C}_{c}^{\infty}(X)\big)$-generic; namely, for every $\varphi\in \mathscr{C}_{c}^{\infty}(X)$ and for $\nu$-a.e. $x\in X$ it holds that
\begin{equation*}\label{eq:ge}
    \frac{1}{t^{m+n-2}\cdot \vol_{m+n-2}(M_{\sigma})}\int_{tM_{\sigma}} \varphi(a.x)\,\dd a \rightarrow \int_X \varphi\,\dd \mu\quad\mbox{as }t\to\infty.
\end{equation*}
This implies that for $\nu$-a.e.\ $x$ the sets $(tM_{\sigma}).x$ equidistribute as $t\to \infty$, and hence, for any fixed $\varepsilon>0$, they must intersect the cuspidal neighborhood $\{x\in X: \delta(x)<\varepsilon\}$ if $t$ is large enough in terms of $\varepsilon$. The conclusion follows from Corollary \ref{cor:sing}(a).
\qed

\smallskip

A modification of Theorem \ref{thm:genericgeneral}, where  $\varphi$ is replaced with a smooth cut-off of the characteristic function of the set $\left\{x\in X: \delta(x)<e^{-R(t)}\right\}$, will yield a proof of Theorem~\ref{thm:main}. In the next section we present a detailed account of this argument.

\section{Proof of Theorem 
\ref{thm:main}}
\label{sec:Diophproof}

\subsection{Construction of a mollifier}
\label{subsec:moll}
In this subsection we construct a mollifier on the group $G=\operatorname{SL}_{m+n}(\mb R)$, which will later be needed to adapt the proof Theorem \ref{thm:special} to sequences of functions $\varphi$ depending on the parameter $t$.

Let $V_{\id}$ be a symmetric open neighborhood of $\operatorname{id}$ in $G$ that admits a parametrization $\psi:B_{1}(\bs 0)\to V_{\id}$, where the ball $B_{1}(\bs 0)$ is contained in $\mb R^{{(m+n)^2-1}}$.  We may assume that $V_{\id}$ is so small that the right-invariant distance $\operatorname{d}_{G}$ induced by a fixed scalar product on $\mf{sl}_{m+n}(\mb R)$ is Lipschitz-equivalent to any matrix norm on $G$ within $V_{\id}$. From now on, we fix one such matrix norm $\|\cdot\|$ and, with an abuse of notation, for $r>0$ we let $B_r(\id)$ denote the ball $\{g\in G:\|g-\id\|<r\}$. We also select a positive number $r_0>0$ such that the ball $B_{r_0}(\operatorname{id})\subset V_{\id}$.

\medskip

We recall that the Haar measure on $G$ is given by a smooth volume form. This can be seen as follows. Fix an orthonormal basis $\{\alpha_1,\dotsc,\alpha_{(m+n)^2-1}\}$ of $\mf{sl}_{m+n}(\mb R)$ and consider the alternating form $\omega_{\operatorname{id}}$ on $\mf{sl}_{m+n}(\mb R)$ given by
$$\omega_{\operatorname{id}}\left(A\left(\alpha_1,\dotsc,\alpha_{(m+n)^2-1}\right)\right):=\det A$$
for any $A\in{\mb R^{\left((m+n)^2-1\right)\times \left((m+n)^2-1\right)}}$. Then $\omega_{\operatorname{id}}$ can be extended to an invariant volume form $\omega$ on $G$ by setting
$$\omega_{g}:=\left(g^{-1}\right)^{*}\omega_{\operatorname{id}}.$$
As a consequence of this, for the smooth form $\omega$ it holds that $g^{*}\omega=\omega$ for all $g\in G$. In particular, if $\varphi$ is a smooth function on $G$, we have that 
$$\int_{G}(\varphi\circ g^{-1})\cdot \omega=\int_{gG}g^{*}\big((\varphi\circ g^{-1})\cdot \omega\big)=\int_{G}\varphi\cdot g^{*}\omega=\int_{G}\varphi\cdot\omega,$$
showing that $\omega$ induces a Haar measure. In {a} chart, we have
\begin{equation}
\label{eq:density}
\begin{aligned}
 &\ \omega_{\psi(\bs x)}\left(\frac{\partial\psi}{\partial x_1}(\bs x),\dotsc,\frac{\partial\psi}{\partial x_{(m+n)^2-1}}(\bs x)\right)\\ =  &\ \omega_{\operatorname{id}}\left(\psi(\bs x)^{-1}\frac{\partial\psi}{\partial x_1}(\bs x),\dotsc,\psi(\bs x)^{-1}\frac{\partial\psi}{\partial x_{(m+n)^2-1}}(\bs x)\right) \\
 =  &\ \det\left(\left(\psi(\bs x)^{-1}\frac{\partial\psi}{\partial x_1}(\bs x),\dotsc,\psi(\bs x)^{-1}\frac{\partial\psi}{\partial x_{(m+n)^2-1}}(\bs x)\right)^{t}\cdot\left(\alpha_1,\dotsc,\alpha_{(m+n)^2-1}\right)\right) \\
=: &  \ f(\alpha_1,\dotsc,\alpha_{(m+n)^2-1},\bs x).
\end{aligned}
\end{equation}
Then, by definition of integration of a volume form, for any smooth function $\varphi$ supported on $B_{r_0}(\operatorname{id})$ it holds that
$$\mu(\varphi)=\int_{G}\varphi\cdot \omega=\int_{B_{1}(\bs 0)}\varphi\big(\psi(\bs x)\big)\cdot f(\alpha_1,\dotsc,\alpha_{(m+n)^2-1},\bs x)\,\dd\bs x,$$
proving that the Haar measure $\mu$ induced by $\omega$ is locally the Lebesgue measure up to a density function $f(\alpha_1,\dotsc,\alpha_{(m+n)^2-1},\bs x)$. Note that the function $f$ is never $0$ in view of \eqref{eq:density}.

\medskip

Let $0<r<r_0$ and let $\rho$ be a fixed bump function of integral $1$ supported on $B_{1}(\bs 0)$. For any $\bs x\in B_{1}(\bs 0)$ we define
\begin{equation}
\label{eq:rho}
\rho_r\big(\psi(\bs x)\big):=r^{-(m+n)^2+1}\frac{\rho\left(r^{-1}\bs x\right)}{f(\alpha_1,\dotsc,\alpha_{(m+n)^2-1},\bs x)}.
\end{equation}

\begin{lem}
\label{lem:mollifierprop}
The function $\rho_r$ defined in \eqref{eq:rho} is a smooth positive and compactly supported function on $B_{r}(\operatorname{id})$ such that
\begin{itemize}
    \item[$(i)$] For any choice of directions $\beta_1,\dotsc,\beta_k\in\mf{sl}_{m+n}(\mb R)$ with $\|\beta_i\|=1$ it holds that
    $$\|D_{\beta_k}\circ\dotsb\circ D_{\beta_1}\rho_r\|_{\infty}\ll_{k} r^{-(m+n)^2+1-k}.$$
    \item[$(ii)$] For all values of $r>0$ it holds that $\int_{G}\rho_r\,\dd\mu=1$.\vspace{2mm}
    \item[$(iii)$] There exists a constant $C_1>0$ such that $\operatorname{supp}(\rho_r)\subset B_{C_1r}(\id)$ for any $r>0$.
\end{itemize}
\end{lem}

\begin{proof}
Let $\alpha\in\mf{sl}_{m+n}(\mb R)$ with $\|\alpha\|=1$. For any $\bs x\in B_{1}(\bs 0)$, we define
$$\gamma_{\alpha}(\bs x,t):=\psi^{-1}\big(\exp(t\alpha)\psi(\bs x)\big).$$
This is a smooth function, given that $\psi^{-1}$ can always be extended to a smooth function on a neighbourhood of $\operatorname{id}$ in $\mb R^{(m+n)^2}$ by the inverse function theorem. Then for every $\bs x\in B_{1}(\bs 0)$ one has that
\begin{multline*}
D_{\alpha}\rho_r\big(\psi(\bs x)\big)=r^{-(m+n)^2+1}\frac{\dd}{\dd t}\rho\bigg(\frac{r^{-1}\cdot\gamma_{\alpha}(\bs x,t)}{f(\alpha_1,\dotsc,\alpha_{(m+n)^2-1},\gamma_{\alpha}\big(\bs x,t)\big)}\bigg)_{|t=0} \\
=r^{-(m+n)^2}\sum_{i=1}^{(m+n)^2-1}\frac{\partial\rho}{\partial x_i}\left(\frac{r^{-1}\cdot\bs x}{f(\alpha_1,\dotsc,\alpha_{(m+n)^2-1},\bs x)}\right) \\
\cdot\frac{\dd}{\dd t}\left(\frac{\gamma_{\alpha}(\bs x,t)}{f\big(\alpha_1,\dotsc,\alpha_{(m+n)^2-1},\gamma_{\alpha}(\bs x,t)\big)}\right)_{|t=0}.
\end{multline*}
Using the fact that all functions are smooth and defined on compact sets, we deduce that $\|D_{\alpha}\rho\|_{\infty}\ll r^{-(m+n)^2}$. Higher order derivatives are treated in the same way, hence $(i)$ follows.

\medskip

As for $(ii)$, we see that
\begin{align*}
& \int_{G}\rho_r\,\dd\mu=\int_{B_{r}(\bs 0)}\rho_r\big(\psi(\bs x)\big)\cdot f(\alpha_1,\dotsc,\alpha_{(m+n)^2-1},\bs x)\,\dd\bs x \\
& =\int_{B_{r}(\bs 0)}r^{-(m+n)^2+1}\frac{\rho\left(r^{-1}\bs x\right)}{f(\alpha_1,\dotsc,\alpha_{(m+n)^2-1},\bs x)}\cdot f(\alpha_1,\dotsc,\alpha_{(m+n)^2-1},\bs x)\,\dd\bs x \\
&=\int_{B_{1}(\bs 0)}\rho(\bs y)\,\dd\bs y=1.
\end{align*}

\medskip

For part $(iii)$, we observe that $\operatorname{supp}(\rho\big(r\cdot \bs x)\big)\subset B_{r}(\bs 0)$. Since $\psi$ has bounded derivative over $B_{1}(\bs 0)$, by the Mean Value Theorem there there exists a constant $C_1>0$, only depending on the choice of $\psi$ such that $\psi(\operatorname{supp}\big(\rho_r)\big)\subset B_{C_1r}(\id)$. This concludes the proof.
\end{proof}

\subsection{Annuli around the cusp}

We aim to extend the argument used to prove Theorem \ref{thm:special} to cuspidal neighborhoods in the space $X$ that shrink as the time $t$ increases. To do so, we need to smooth out the characteristic {functions} of such neighborhoods and to make the dependence on the parameter $t$ explicit in all our estimates. This is precisely the goal of {this} subsection.

In what follows, we make use of the notation introduced in Subsection \ref{subsec:moll}. In particular, a matrix norm $\|\cdot\|$ and a right-invariant distance $\dd_G$ are chosen on $G$. The real number $r_0$ is fixed and small enough so that the ball $B_{r_0}(\id):=\{g\in G:\|g-\id\|<r_0\}$ is contained in a symmetric open neighborhood $V_{\id}$ of the identity in $G$ where the matrix norm and the right-invariant distance $\dd_G$ on $G$ are Lipschitz equivalent.

Given a continuous function $R:[0,\infty)\to[0,\infty)$ and $t>0$ we define
\begin{equation}
\label{eq:at}
\Omega_{t}:=\left\{x\in X:e^{-R(t)-1}\leq \delta(x)<e^{-R(t)}\right\}.    
\end{equation}
Recall that the element $x$ is identified here with a lattice in $\mb R^{m+n}$ and that $\delta$ denotes the length of the shortest vector with respect to the supremum norm on $\mb R^{m+n}$, henceforth also denoted by $\|\cdot\|$. Fix a small radius $0<r<r_0$ and for $t>0$ put
\begin{equation}
\label{eq:varphitr}
\varphi_{t,r}(x):=\rho_{r}\ast \chi_{\Omega_t}(x)=\int_{G}\rho_r(h)\cdot \chi_{\Omega_t}(h^{-1}x)\,\dd\mu(h),    
\end{equation}
where $\rho_r$ is the mollifier introduced in the previous subsection.
\begin{lem}
\label{lem:phitrproperties}
Let $\|\cdot \|_{\infty}$ denote the supremum norm on the space $\mathscr{C}^{\infty}_c(X)$ and let $\|\cdot\|_{h}$ denote the $h$-norm on the space $\operatorname{L}^{h}(X,\mu)$. Then for any $t>0$:
\begin{itemize}
    \item[$(i)$] $\mu(\varphi_{t,r})=\mu(\Omega_t)$;\vspace{2mm}
    \item[$(ii)$] for all $k\geq 0$ and all choices of directions $\beta_1,\dotsc,\beta_k\in\mf{sl}_{m+n}(\mb R)$, with $\|\beta_i\|=1$, it holds that
    $$\|D_{\beta_k}\circ\dotsb\circ D_{\beta_1}\varphi_{t,r}\|_{\infty}\ll_{k} r^{-k},$$
where the expression $D_{\beta_k}\circ\dotsb\circ D_{\beta_1}\varphi_{t,r}$ stands for $\varphi_{t,r}$ when $k=0$. \vspace{2mm}
\item[$(iii)$] for all $h\geq 1$, all $k\geq 0$, and all choices of directions $\beta_1,\dotsc,\beta_k\in\mf{sl}_{m+n}(\mb R)$, with $\|\beta_i\|=1$, it holds that
    $$\|D_{\beta_k}\circ\dotsb\circ D_{\beta_1}\varphi_{t,r}\|_{h}\ll_{k} r^{-k}\cdot\mu(\Omega_t)^{1/h},$$
where the expression $D_{\beta_k}\circ\dotsb\circ D_{\beta_1}\varphi_{t,r}$ stands for $\varphi_{t,r}$ when $k=0$.
\end{itemize}
\end{lem}

\begin{proof}
For $(i)$, we see that
\begin{equation*}
\int_{X}\int_{G}\rho_r(g)\cdot \chi_{\Omega_t}(g^{-1}x)\,\dd\mu(g)\,\dd\mu(x) \\
=\int_{G}\rho_{r}\,\dd\mu\cdot \int_{X}\chi_{\Omega_t}\,\dd\mu=1\cdot \mu(\Omega_t),
\end{equation*}
where we used Lemma \ref{lem:mollifierprop} part $(ii)$. For $(ii)$, we observe first that
\begin{multline*}
D_{\beta_1}\varphi_{t,r}(x)=\frac{\dd}{\dd t}\int_{G}\rho_r(g)\cdot \chi_{\Omega_t}(g^{-1}\exp(t\beta_1)x)\,\dd\mu(g)_{|t=0} \\
=\frac{\dd}{\dd t}\int_{G}\rho_r(\exp(-t\beta_1)g)\cdot \chi_{\Omega_t}(g^{-1}x)\,\dd\mu(g)_{|t=0}=\int_{G}D_{\beta_1}\rho_r(g)\cdot \chi_{\Omega_t}(g^{-1}x)\,\dd\mu(g),
\end{multline*}
where the derivative can be taken under the integral sign, since for $0<t<1$ the sets $\exp(t\beta_1)B_{r}(\id)$ all lie in a larger set of finite measure. Then we deduce that
\begin{equation}
\label{eq:deriv}
\left|D_{\beta_1}\varphi_{t,r}(x)\right|\leq \mu\big(B_{C_1r}(\id)\big)\cdot\|D_{\beta_1}\rho_r\|_{\infty}\ll r^{-1},    
\end{equation}
where we have used  parts $(i)$ and $(iii)$ of Lemma \ref{lem:mollifierprop}.
When deriving in more than one direction the computation is analogous.

Finally, to prove $(iii)$, we see that for any $\phi_1\in L^1(G)$ and any $\phi_2\in L^{\infty}(X)$ it holds 
\begin{align*}
& \int_{X}(\phi_1\ast \phi_2)^h\,\dd\mu \\
& =\int_{X}\int_{G}\dotsb\int_{G}\phi_1(g_1)\dotsm\cdot\phi_1(g_h)\cdot\phi_2(g_{1}^{-1}x)\dotsm\phi_2(g_h^{-1}x)\,\dd\mu(g_1)\dotsb\,\dd\mu(g_h)\,\dd\mu(x) \\
& \leq \int_{G}\dotsb\int_{G}|\phi_1(g_1)|\dotsm|\phi_1(g_h)|\cdot\left|\int_{X}\phi_2(g_{1}^{-1}x)\dotsm\phi_2(g_h^{-1}x)\,\dd\mu(x)\right|\,\dd\mu(g_1)\dotsb\,\dd\mu(g_k) \\
&\leq \|\phi_1\|_{1}^h\cdot\|\phi_2\|_h^{h},    
\end{align*}
where in the last step we used H\"older's inequality. Combining this observation with the first inequality in \eqref{eq:deriv}, we deduce that
\begin{multline*}\|D_{\beta_k}\circ\dotsb\circ D_{\beta_1}\varphi_{t,r}\|_h\leq \|D_{\beta_k}\circ\dotsb\circ D_{\beta_1}\rho_r\|_{1}\cdot\|\chi_{\Omega_t}\|_h \\
\leq \mu\big(B_{C_1r}(\id)\big)\cdot\|D_{\beta_k}\circ\dotsb\circ D_{\beta_1}\rho_r\|_{\infty}\cdot\|\chi_{\Omega_t}\|_h\ll_k r^{(m+n)^2-1}\cdot r^{-(m+n)^2+1-k}\mu(\Omega_t)^{1/h}.
\end{multline*}
where we have used parts $(i)$ and $(iii)$ of Lemma \ref{lem:mollifierprop}.
\end{proof}

\begin{lem}
\label{lem:approx}
Assume that the chosen matrix norm on $G$ is the supremum norm. Then, provided $r_0\leq 1$, for any $r<r_0$ and $t>0$ we have that
$$\operatorname{supp}(\varphi_{t,r})\subset\left\{x\in X:\frac{1}{2C}e^{-R(t)-1}\leq \delta(x)<2e^{-R(t)}\right\},$$
where $C$ is the Lipschitz constant in the equivalence between the matrix norm $\|\cdot\|$ and the right-invariant distance $\dd_G$ within $B_{r_0}(\id)$.
\end{lem}

\begin{proof}
In what follows, with an abuse of notation, we denote by $\|\cdot\|$ both the supremum norm on $\mb R^{m+n}$ and the supremum matrix norm on $G$. Let $x=g\Gamma$, where $g\in G$ and $\Gamma=\operatorname{SL}_{m+n}(\mb Z)$. Then, by definition, we have that
$$\varphi_{t,r}(g\Gamma)=\int_{G}\rho_r(h)\cdot \chi_{\Omega_t}(h^{-1}g\Gamma)\,\dd\mu(h).$$
Hence $g\Gamma\in\operatorname{supp}(\varphi_{t,r})$ if and only if there exists $h\in B_{r}(\id)$ such that $h^{-1}g\Gamma\in \Omega_t$. 
Now, let $g\Gamma\in\operatorname{supp}(\varphi_{r,t})$ and let $\bs w\in\mb Z^{m+n}\setminus\{\bs 0\}$ such that $\|h^{-1}g\bs w\|<e^{-R(t)}$. Then it must be
$$e^{-R(t)}>\|h^{-1}g\bs w\|\geq \|h\|^{-1}\|g\bs w\|$$
(this holds for all matrix norms), whence
$$\|g\bs w\|\leq \|h\|e^{-R(t)}\leq (\|\id\|+\|\id -h\|)e^{-R(t)}\leq (1+r)e^{-R(t)},$$
showing that $\delta(g\Gamma)<2e^{-R(t)}$. On the other hand, we have $\delta(h^{-1}g\Gamma)\geq e^{-R(t)-1}$. Therefore for all $\bs w'\in\mb Z^{m+n}\setminus\{\bs 0\}$ it holds that
$$\|h^{-1}\|\cdot\|g\bs w'\|\geq \|h^{-1}g\bs w\|\geq e^{-R(t)-1},$$
whence
$$\|g\bs w'\|\geq \frac{1}{C(1+r)}e^{-R(t)-1},$$
where we used the fact that balls with respect to the invariant distance are symmetric and that the matrix norm $\|\cdot\|$ induces a distance that is Lipschitz equivalent (with constant $C$) to the invariant distance. This proves the other inequality.
\end{proof}

\subsection{Equidistribution of hyperplanes}

We now apply Lemma \ref{lem:4mom} to the functions $\varphi_{t,r}$ constructed in the previous section, to prove a modified version of Lemma \ref{lem:BC1}.

\begin{lem}
\label{lem:intersection} {Assume that $\max\{m,n\} > 1$.}
Let $0<\sigma<\frac1{\max\{m,n\}}$ and let
\begin{multline*}
J:=\left\{\bs s:=(s_1,\dotsc,s_{m-1},z_1,\dotsc,z_{n-1})\in [\sigma,1-\sigma]^{m+n-2}:\right. \\
\left.s_1+\dotsb +s_{m-1}\leq 1-\sigma,\ z_1+\dotsb +z_{n-1}\leq 1-\sigma\right\}.
\end{multline*}
Consider the global parametrization $\gamma:J\to M_{\sigma}\subset \mf a$  of 
$M_{\sigma}$ (defined in \eqref{eq:Esigma}) given by
\begin{multline*}
(s_1,\dotsc,s_{m-1},z_1,\dotsc,z_{n-1}) \\
\mapsto(s_1,\dotsc,s_{m-1},1-(s_1+\dotsb+s_{m-1}),-z_1,\dotsc,-z_{n-1},(z_1+\dotsb z_{n-1})-1),
\end{multline*}
and let $\varphi_{t,r}$ be as in \eqref{eq:varphitr} for some $r<r_0$. Finally, assume that for a fixed even $h\in\mb N$ the condition
\begin{equation}
\label{eq:measure}
\liminf_{t\to\infty}\frac{\log \mu(\Omega_t)}{\log t}> -(m+n-2)+\frac{m+n-1}{h}
\end{equation}
is satisfied, where $\Omega_t$ is defined in \eqref{eq:at}. Then for almost every $Y\in[0,1)^{m\times n}$ and all large enough values of $t\in\mb N$ (depending on $h,m,n,$ and $Y$) it holds that
\begin{equation}
\label{eq:thesis}
\int_{J}\varphi_{t,r}\left(\exp\big(t\gamma(\bs s\big)x_{Y}\right)\dd\bs s>\operatorname{Leb}(J)\cdot \mu(\Omega_t)/2.    
\end{equation}   
\end{lem}

\begin{proof}
Fix a constant $\tau$ such that $\tau<\left(\frac{\textup{Leb}(J)}{2}\right)^{h}$.
For $t\in\mb N$ consider the sets
$$S_{t}:=\left\{Y\in [0,1)^{mn}:\left(\int_{J}\varphi_{t,r}\left(\exp\big(t\gamma(\bs s)\big)x_{Y}\right)\dd\bs s-\textup{Leb}(J)\cdot\mu(\Omega_t)\right)^h>\tau\cdot \mu(\Omega_t)^{h}\right\}.$$
Observe that Condition \eqref{eq:4mix} is satisfied for all $t>0$ and all $a\in {tM_{\sigma}}$ as a result of Theorem \ref{BG23} (for a more detailed explanation see the argument used in Subsection \ref{subsec:2.2imp1.4}). Then, by Markov's Inequality and Lemma \ref{lem:4mom} applied with ${b}=t^{-1+\delta}$ ($\delta$ positive and only depending on $m,n,$ and $h$), we find that for all sufficiently large $t\in\mb N$ it holds that
\begin{align}
\label{eq:summability}
\nu(S_t)&\leq \frac{1}{\tau\cdot \mu(\Omega_t)^{h}}\cdot\int_{[0,1)^{mn}}\left(\int_{J}\varphi_{t,r}\left(\exp\big(t\gamma(\bs s)\big)x_{Y}\right)\dd\bs s-\textup{Leb}(J)\cdot\mu(\Omega_t)\right)^h\,\dd\nu \\[2mm]\nonumber
&\ll_{h,r,\tau} \frac{t^{(-1+\delta)(m+n-2)(h-\lfloor{h/2}\rfloor)}\cdot \mu(\Omega_{t})^{h/2}+t^{(-1+\delta)(m+n-2)(h-1)}}{\mu(\Omega_t)^{h}}\\[2mm] \nonumber
&= t^{(-1+\delta)(m+n-2)(h/2)}\cdot \mu(\Omega_{t})^{-h/2}+t^{(-1+\delta)(m+n-2)(h-1)}\cdot\mu(\Omega_t)^{-h}. \nonumber
\end{align}
Here, we used Lemma \ref{lem:phitrproperties} to estimate the constants $A_{\varphi}$ and $B_{\varphi}$ for $\varphi=\varphi_{t,r}$, and the fact that $h$ is even.  Note that the quantities $\lambda_{\sup}$ and $\lambda_{\inf}$ are positive and finite 
{in view of Lemma \ref{lem:lambdainfsup} and Remark \ref{rem:reduction}}. 
Now, if the sets $S_t$ have summable measures, by the Borel--Cantelli Lemma  we have that for almost every $Y\in [0,1)^{mn}$ and all but finitely many $t\in\mb N$ (depending on $Y$) it holds that
\begin{equation}
\label{eq:BC}
\left|\int_{J}\varphi_{t,r}\left(\exp\big(t\gamma(\bs s)\big)x_{Y}\right)\dd\bs s-\textup{Leb}(J)\cdot\mu(\Omega_t)\right|\leq \tau^{1/h}\cdot \mu(\Omega_t).   
\end{equation}
By the choice of $\tau$, this implies the conclusion. To prove \eqref{eq:thesis}  it is therefore enough to impose that the following system of inequalities hold
$$\begin{cases}
(-1+\delta)(m+n-2)\dfrac{h}{2}-\dfrac{\log \mu(\Omega_t)}{\log t}\dfrac{h}{2}<-1 \\[3mm]
(-1+\delta)(m+n-2)(h-1)-\dfrac{\log \mu(\Omega_t)}{\log t}h<-1
\end{cases},$$
which ensures the summability of the measures in \eqref{eq:summability}.
The system is equivalent to
$$
\dfrac{\log \mu(\Omega_t)}{\log t}>\max\left\{(-1+\delta)(m+n-2)+\frac{2}{h},\frac{(-1+\delta)(m+n-2)(h-1)}{h}+\frac{1}{h}\right\}.$$
Since $\delta>0$ can be chosen arbitrarily small, this may also be rewritten as 
$$\begin{aligned}\liminf_{t\to\infty}\dfrac{\log \mu(\Omega_t)}{\log t}&>\max\left\{-(m+n-2)+\frac{2}{h},\frac{-(m+n-2)(h-1)}{h}+\frac{1}{h}\right\}\\ &=  - (m+n-2) + {\max\left\{\frac{2}{h},\frac{m+n-2}{h}+\frac{1}{h}\right\}}.\end{aligned}$$
{Since $m+n>2$}, the second term  dominates. {Thus} the conditions in the system are ensured by \eqref{eq:measure}. 
\end{proof}

\subsection{Completion of the proof}

We now complete the proof of Theorem \ref{thm:main}.

For $\lambda\geq 0$ set
$$\psi_{\lambda}(x):=
\begin{cases}
e^{-1} & \mbox{if }x<e \\
x^{-1}(\log x)^{-\lambda} & \mbox{if }x\geq e
\end{cases}.$$
Let $t_{\lambda}>0$ be the unique value of $t$ such that $t-nR_{\lambda}(t)=e$. Then for all $t\geq t_{\lambda}$, the function $R_{\lambda}$ associated to $\psi_{\lambda}$ through Lemma \ref{lem:C1} is defined via the equation
$$\frac{e^{-t+nR_{\lambda}(t)}}{\big(t-nR_{\lambda}(t)\big)^{\lambda}}=\psi_{\lambda}\left(e^{t-nR_{\lambda}(t)}\right)=e^{-t-mR_{\lambda}(t)}.$$
From this, we deduce that 
\begin{equation}
\label{eq:findR}
e^{-(m+n)R_{\lambda}(t)}=\frac{1}{\big(t-nR_{\lambda}(t)\big)^{\lambda}}    
\end{equation}
for all $t\geq t_{\lambda}$. Hence, again in the range $t\geq t_{\lambda}$, we have that $0\leq R_{\lambda}(t)\ll_{m,n}\log t$. Consequently,
$$e^{-(m+n)R_{\lambda}(t)}=\frac{1}{\big(t-nR_{\lambda}(t)\big)^{\lambda}}\sim t^{-\lambda},$$
where the symbol $\sim$ is used to denote the fact that both the inequalities $\ll$ and $\gg$ apply (the implicit constants here depend only on $m$ and $n$).

\medskip

Let us now assume that $\lambda<m+n-2$. Fix $\varepsilon>0$ so that $\lambda+\varepsilon<m+n-2$. We aim to show that \eqref{eq:hyper} is satisfied for all large enough values of $t$.
For $\Omega_t$ as in \eqref{eq:at} relative to the function $R_{\lambda+\varepsilon}$, we have that
$$\mu(\Omega_t)\sim e^{-(m+n)R_{\lambda+\varepsilon}(t)}\sim t^{-\lambda-\varepsilon},$$
where we used Lemma \ref{lem:approx}. Let $0<\sigma<\frac1{\max\{m,n\}}$. Then, on choosing $h\in\mb N$ so large that
$$-\lambda-\varepsilon>-(m+n-2)+\frac{m+n-1}{h},$$
Condition \eqref{eq:measure} is satisfied, and we deduce by Lemma \ref{lem:approx} and Lemma \ref{lem:intersection} that for almost every $Y\in [0,1)^{mn}$ and for all sufficiently large $t\in \mb N$ there exists $\bs s\in(\sigma,1-\sigma)^{m+n-2}$ such that
$$\delta\left(\exp\left(t\gamma(\bs s)\right)x_{Y}\right)<2e^{-R_{\lambda+\varepsilon}(t)}.$$ 
Thus, interpolating between integers and using the fact that $|R_{\lambda+\varepsilon}'(t)|\leq 1$ (as a result of \eqref{eq:findR}), we conclude that for all sufficiently large $t\in\mb R$ it holds that
$$\delta\left(\exp\left(t\gamma(\bs s)\right)x_Y\right)<2e^{-R_{\lambda+\varepsilon}(t)+1}<e^{-R_{\lambda}(t)}$$ 
for some $\bs s\in M_{\sigma}$. Corollary \ref{prop:corr} (see \eqref{eq:hyper}) then implies that $Y\in \UA^{\times}_{m,n}(\psi_{1,\lambda})$, provided we show that
$$t M_{\sigma}\subset C_{R_{\lambda}}(t)$$
for all sufficiently large $t>0$. This holds, since for any $a\in M_{\sigma}$ and any $j=m+1,\dotsc,m+n$ one has that
$$-ta_{j}\geq t\sigma>R_{\lambda}(t)$$
if $t$ is large enough (recall that $R(t)\ll_{m,n}\log t$). Thus, if $\lambda<m+n-2$, it must be $\textup{Leb}\left(\UA^{\times}_{m,n}(\psi_{1,\lambda})\cap[0,1)^{mn}\right)=1.$

\medskip

Finally, we are left to show that the sets $\UA ^{\times}_{m,n}(\psi_{1,\lambda})$ have measure $0$ for $\lambda>m+n-2$. We recall that
$$\UA^{\times}_{m,n}(\psi_{1,\lambda}):=\bigcup_{T_0\geq 1}\bigcap_{T\geq T_0}\mc S^{\times}_{m,n}(\psi,T).$$
Then it is enough to show that
$$\operatorname{Leb}\left(\mc S^{\times}_{m,n}(\psi_{\lambda,1},T)\right)\to 0$$
as $T\to \infty$. We note that
\begin{multline*}
\mc S^{\times}_{m,n}(\psi_{\lambda,1},T)\subset \\
\bigcup_{\Pi_{+}(\bs q)<T}\bigcup_{|\bs p|<|\bs q|}\left\{Y\in[0,1)^{mn}:\prod_{i=1}^{m}|Y_i\bs q-p_i|<\frac{1}{T(\log T)^{\lambda}},\ |Y_i\bs q-p_i|\leq\frac{1}{2}\right\},
\end{multline*}
where $Y_i$ denotes the $i$-th row of the matrix $Y$.
Without loss of generality, we may assume that $q_1=|\bs q|$ (if not, the argument is very similar). Then the change of variable
$$\begin{cases}
Y_{i,1}\mapsto Y_i\bs q & \mbox{for }i=1,\dotsc,m \\
Y_{i,j}\mapsto Y_{i,j} & \mbox{for }j\neq 1
\end{cases}
$$
has determinant $|\bs q|^m$. Since the set
$$\left\{\bs x\in\mb R^m: \prod_{i=1}^{m}|x_i|<\varepsilon,\ |x_i|\leq\frac{1}{2}\right\}$$
has volume
$$2^m\varepsilon\cdot \left[\log\left(2^{-m}\varepsilon^{-1}\right)^{m-1}+1\right],$$
we deduce that
$$\operatorname{Leb}(\mc S^{\times}_{m,n}\big(\psi_{\lambda,1},T)\big)\ll_{m,n}\sum_{\Pi_{+}(\bs q)<T}\sum_{|\bs p|\leq |\bs q|+1}\frac{1}{|\bs q|^m}\cdot \frac{1}{T(\log T)^{\lambda-(m-1)}}\ll_{m,n} (\log T)^{m+n-2-\lambda}.$$
Thus the proof is concluded.

\section{Proof of Proposition \ref{prop:dno}}
\label{sec:dnoproof}

{We now prove Proposition \ref{prop:dno} by adapting an argument presented in \cite[Theorem 3A]{schmidt}. For $c>0$, define $T_c:=\max\{cm^m,n^{n(n-1)}\}$.
For $Y\in\mb R^{m\times n}$, $c>\frac{m! n!}{m^m n^n}$ and $T> T_c$ consider the set 
    \begin{equation}
{\Xi(
Y,T,c):= \\
\left\{(\bs{x},\bs{y})\in \R ^{m+n}:\begin{cases}
\sum_{i=1}^{m}|Y_{i}\bs{y}-x_{i}|\leq m(\frac{c}{T})^{1/m} \\
\sum_{j=1}^{n}|y_j|\leq nT^{1/n}
\end{cases}
\right\}
\nonumber}.
\end{equation}
This is a convex and symmetric set whose volume is given by $$\frac{2^{m+n}m^m n^n c}{m! n!}\geq 2^{m+n}.$$
Hence, by Minkowski's Convex Body Theorem, $$\Xi(
Y,T,c)\cap \Z^{m+n}\neq \{\bs 0\}.$$
Let $0\neq(\bs{p},\bs{q})\in \Xi(
Y,T,c)\cap \Z^{m+n}$. If $\bs{q}=0$, then 
$$1\leq\sum_{i=1}^m |p_i|\leq m({c}/{T})^{1/m}$$
which gives $T\leq cm^m$. Hence $\bs{q}\neq 0$. Suppose $1\leq k\leq n$ be such that $q_i\neq 0$ for $1\leq i\leq k$ and $q_i=0$ for $k+1\leq i\leq n$. Then we have
$$\sum_{j=1}^{k}|q_j|\leq nT^{1/n}$$
Using the AM--GM inequality, which states for any nonnegative real numbers $x_1,\dots,x_n$ one has that
$$ (x_1\dots x_n)^{1/n}\leq \frac{x_1+\dots+x_n}{n},$$
we deduce that
$$\left(\prod_{j=1}^k|q_j|\right)^{\frac{1}{k}}\leq \frac{1}{k}\sum_{j=1}^k|q_j|\leq \frac{n}{k}T^{1/n}.$$
Hence
\begin{equation}
\label{eq:prod1}
  \prod_{j=1}^n \max\{1,|q_j|\}= \prod_{j=1}^k |q_j|\leq \left(\frac{n}{k}\right)^k T^{k/n}\leq T.  
\end{equation}
The last inequality follows since $T\geq n^{n(n-1)}$. Again, by the AM--GM inequality, we have that
\begin{equation}
\label{eq:prod2}
\left(\prod_{i=1}^m|Y_i\bs{q}-p_i|\right)^{1/m}\leq \frac{1}{m}\sum_{i=1}^m |Y_i\bs{q}-p_i|\leq \left(\frac{c}{T}\right)^{1/m}.
\end{equation}
Finally, \eqref{eq:prod1} and \eqref{eq:prod2} together imply that $Y\in \mc S^{\times}_{m,n}(c\psi_1,T)$ for every $T\geq T_c$, since, by definition,
\begin{equation}
{\mc{S}_{m,n}^{\times}(
c\psi_1,T):= \\
\left\{Y\in \mb{R}^{m\times n}:\exists\,\bs{p}\in\mb{Z}^{m},\, \bs{q}\in\mb{Z}^{n}\setminus\{\bs{0}\}\mbox{ s.t.}\begin{cases}
\prod_{i=1}^{m}|Y_{i}\bs{q}-p_{i}|< c/T \\
\prod_{i=1}^{n}\max\{1,|q_i|\}< T
\end{cases}
\right\}\nonumber}.
\end{equation}
Hence $Y\in \UA^{\times}_{m,n}(c\psi_1) $ for $c>\frac{m! n!}{m^m n^n}$.}
\bibliographystyle{alpha}
\bibliography{References}

\begin{thebibliography}{KKLM17}

\bibitem[Bad13]{Bad13}
D.~Badziahin.
\newblock On multiplicatively badly approximable numbers.
\newblock {\em Mathematika}, 59(1):31--55, 2013.

\bibitem[BB22]{BB22}
D.~Badziahin and Y.~Bugeaud.
\newblock Multiplicative {{\(p\)}}-adic approximation.
\newblock {\em Mich. Math. J.}, 71(1):121--143, 2022.

\bibitem[BEG20]{BEG20}
M.~Bj{\"o}rklund, M.~Einsiedler, and A.~Gorodnik.
\newblock Quantitative multiple mixing.
\newblock {\em J. Eur. Math. Soc. (JEMS)}, 22(5):1475--1529, 2020.

\bibitem[BG20]{BG20new}
M.~Bj\"{o}rklund and A.~Gorodnik.
\newblock Central limit theorems for group actions which are exponentially
  mixing of all orders.
\newblock {\em J. Anal. Math.}, 141(2):457--482, 2020.

\bibitem[BG23]{BG23}
M.~Bj{\"o}rklund and A.~Gorodnik.
\newblock Effective multiple equidistribution of translated measures.
\newblock {\em Int. Math. Res. Not.}, 2023(1):210--242, 2023.

\bibitem[Bou86]{Bourgain}
J.~Bourgain.
\newblock Averages in the plane over convex curves and maximal operators.
\newblock {\em J. Anal. Math.}, 47:69--85, 1986.

\bibitem[Bug09]{Bug09}
Y.~Bugeaud.
\newblock Multiplicative {Diophantine} approximation.
\newblock In {\em Dynamical systems and Diophantine approximation}, pages
  105--125. Paris: Soci{\'e}t{\'e} Math{\'e}matique de France, 2009.

\bibitem[BV11]{BV11}
D.~Badziahin and S.~Velani.
\newblock Multiplicatively badly approximable numbers and generalised {Cantor}
  sets.
\newblock {\em Adv. Math.}, 228(5):2766--2796, 2011.

\bibitem[CC16]{CC16}
Y.~Cheung and N.~Chevallier.
\newblock Hausdorff dimension of singular vectors.
\newblock {\em Duke Math. J.}, 165(12):2273--2329, 2016.

\bibitem[Che11]{Che11}
Y.~Cheung.
\newblock Hausdorff dimension of the set of singular pairs.
\newblock {\em Ann. Math. (2)}, 173(1):127--167, 2011.

\bibitem[Dan85]{Dani}
S.G. Dani.
\newblock Divergent trajectories of flows on homogeneous spaces and
  {Diophantine} approximation.
\newblock {\em J. Reine Angew. Math.}, 359:55--89, 1985.

\bibitem[DFSU24]{DFSU}
T.~Das, L.~Fishman, D.~Simmons, and M.~Urba{\'n}ski.
\newblock A variational principle in the parametric geometry of numbers.
\newblock {\em Adv. Math.}, 437:130, 2024.
\newblock Id/No 109435.

\bibitem[DS70]{DS2}
H.~Davenport and W.M. Schmidt.
\newblock Dirichlet's theorem on diophantine approximation. {II}.
\newblock {\em Acta Arith.}, 16:413--424, 1970.

\bibitem[EKL06]{EKL06}
M.~Einsiedler, A.~Katok, and E.~Lindenstrauss.
\newblock Invariant measures and the set of exceptions to {Littlewood}'s
  conjecture.
\newblock {\em Ann. Math. (2)}, 164(2):513--560, 2006.

\bibitem[FK24]{FK24}
R.~Fregoli and D.~Kleinbock.
\newblock On multiplicatively badly approximable vectors.
\newblock Preprint, {\tt https://arxiv.org/abs/2211.04523}, 2024.

\bibitem[Gal62]{Gallagher}
P.~Gallagher.
\newblock Metric simultaneous diophantine approximation.
\newblock {\em J. Lond. Math. Soc.}, 37:387--390, 1962.

\bibitem[Gro38]{Gr38}
A.V. Groshev.
\newblock Une th\'eorème sur les syst\`emes des formes lin\'eaires.
\newblock {\em Dokl. Akad. Nauk SSSR}, 9:151–152, 1938.

\bibitem[Jam09]{Jaming}
P.~Jaming.
\newblock The spherical ergodic theorem revisited.
\newblock {\em Expo. Math.}, 27(3):257--269, 2009.

\bibitem[Jon93]{Jones}
R.L. Jones.
\newblock Ergodic averages on spheres.
\newblock {\em J. Anal. Math.}, 61:29--45, 1993.

\bibitem[KKLM17]{KKLM17}
S.~Kadyrov, D.~Kleinbock, E.~Lindenstrauss, and G.~A. Margulis.
\newblock Singular systems of linear forms and non-escape of mass in the space
  of lattices.
\newblock {\em J. Anal. Math.}, 133:253--277, 2017.

\bibitem[KM99]{KM99}
D.Y. Kleinbock and G.A. Margulis.
\newblock Logarithm laws for flows on homogeneous spaces.
\newblock {\em Invent. Math.}, 138(3):451--494, 1999.

\bibitem[KSY22]{KSY22}
D.~Kleinbock, A.~Str{\"o}mbergsson, and S.~Yu.
\newblock A measure estimate in geometry of numbers and improvements to
  {Dirichlet}'s theorem.
\newblock {\em Proc. Lond. Math. Soc. (3)}, 125(4):778--824, 2022.

\bibitem[KW08]{KW08}
D.~Kleinbock and B.~Weiss.
\newblock Dirichlet's theorem on {Diophantine} approximation and homogeneous
  flows.
\newblock {\em J. Mod. Dyn.}, 2(1):43--62, 2008.

\bibitem[KW18]{KW18}
D.~Kleinbock and N.~Wadleigh.
\newblock A zero-one law for improvements to {Dirichlet}'s theorem.
\newblock {\em Proc. Amer. Math. Soc.}, 146(5):1833--1844, 2018.

\bibitem[Lac95]{Lacey}
M.T. Lacey.
\newblock Ergodic averages on circles.
\newblock {\em J. Anal. Math.}, 67:199--206, 1995.

\bibitem[Lin01]{Lindenstrauss}
E.~Lindenstrauss.
\newblock Pointwise theorems for amenable groups.
\newblock {\em Invent. Math.}, 146(2):259--295, 2001.

\bibitem[Sch80]{schmidt}
W.M. Schmidt.
\newblock {\em Diophantine approximation}, volume 785 of {\em Lecture Notes in
  Mathematics}.
\newblock Springer, Berlin, 1980.

\bibitem[Spr79]{Spr79}
V.G. Sprindzhuk.
\newblock {\em Metric theory of {D}iophantine approximations}.
\newblock Scripta Series in Mathematics. {V}.{H}. {Winston} \& {Sons}; {A}
  {Halsted} {Press} {Book}. {New} {York}--{Toronto}--{London}: {John} {Wiley}
  \& {Sons}. {XIII}, 156 p., 1979.

\bibitem[Ste76]{Stein}
E.M. Stein.
\newblock Maximal functions. {I}: {Spherical} means.
\newblock {\em Proc. Natl. Acad. Sci. USA}, 73:2174--2175, 1976.

\end{thebibliography}
\end{document}